\documentclass{article}
\usepackage{mathrsfs, amsmath, amsthm, amssymb, bm, mathtools, thm-restate}
\usepackage{graphicx, xcolor, tikz}
\usetikzlibrary{calc, shapes}
\usepackage{caption}
\usepackage[labelformat=simple]{subcaption}

\usepackage{array}
\usepackage{cases}
\makeatletter
\def\th@plain{%
  \upshape 
}
\makeatother

\makeatletter
\renewenvironment{proof}[1][\proofname]{\par
  \pushQED{\qed}%
  \normalfont \topsep6\p@\@plus6\p@\relax
  \trivlist
  \item[\hskip\labelsep
        \bfseries
    #1\@addpunct{.}]\ignorespaces
}{%
  \popQED\endtrivlist\@endpefalse
}
\makeatother

\newtheorem{theorem}{Theorem}[section]
\newtheorem{lemma}[theorem]{Lemma}
\newtheorem{corollary}[theorem]{Corollary}

\newtheorem*{conjecture*}{Conjecture}

\newtheorem{claim}{Claim}

\theoremstyle{definition}

\usepackage[left=25mm, top=25mm, bottom=25mm, right=25mm]{geometry}
\usepackage[T1]{fontenc} 
\usepackage{paralist, tablists}
\usepackage[inline]{enumitem}
\usepackage[square, numbers, sort&compress]{natbib}

\usepackage[
  bookmarks=true,%
  bookmarksnumbered=true, 
  bookmarksopen=true, 
  plainpages=false,%
  colorlinks=true, 
  linkcolor=black, 
  citecolor=black,%
  anchorcolor=green,
  urlcolor=black,
  hyperindex=true]{hyperref}

\newcommand{\etal}{et~al.\ }
\newcommand{\ie}{i.e.,\ }
\usepackage{marvosym,wasysym}

\usetikzlibrary{decorations.pathreplacing,decorations.markings}
\tikzset{
  on each segment/.style={
    decorate,
    decoration={
      show path construction,
      moveto code={},
      lineto code={
        \path [#1]
        (\tikzinputsegmentfirst) -- (\tikzinputsegmentlast);
      },
      curveto code={
        \path [#1] (\tikzinputsegmentfirst)
        .. controls
        (\tikzinputsegmentsupporta) and (\tikzinputsegmentsupportb)
        ..
        (\tikzinputsegmentlast);
      },
      closepath code={
        \path [#1]
        (\tikzinputsegmentfirst) -- (\tikzinputsegmentlast);
      },
    },
  },
  mid arrow/.style={postaction={decorate,decoration={
        markings,
        mark=at position .6 with {\arrow[#1]{stealth}}
      }}},
}
\begin{document}

\title{Decomposition of planar graphs with forbidden configurations}
\author{Lingxi Li\thanks{School of Mathematics and Statistics, Henan University, Kaifeng, 475004, P. R. China}\and Huajing Lu\thanks{College of Basic Science, Ningbo University of Finance and Economics, Ningbo, 315000, P. R. China}\and Tao Wang\thanks{Center for Applied Mathematics, Henan University, Kaifeng, 475004, P. R. China. {\tt Email: wangtao@henu.edu.cn}}\and Xuding Zhu\thanks{School of Mathematical Sciences, Zhejiang Normal University, Jinhua, 321004, P. R. China. This research is supported by Grants: NSFC 11971438, U20A2068.}
}
\date{}
\maketitle

\begin{abstract}
A $(d,h)$-decomposition of a graph $G$ is an ordered pair $(D, H)$ such that $H$ is a subgraph of $G$ of maximum degree at most $h$ and $D$ is an acyclic orientation of $G-E(H)$ with maximum out-degree at most $d$. In this paper, we prove that for $l \in \{5, 6, 7, 8, 9\}$, every planar graph without $4$- and $l$-cycles is $(2,1)$-decomposable. As a consequence, for every planar graph $G$ without $4$- and $l$-cycles, there exists a matching $M$, such that $G - M$ is $3$-DP-colorable and has Alon-Tarsi number at most $3$. In particular, $G$ is $1$-defective $3$-DP-colorable, $1$-defective $3$-paintable and 1-defective 3-choosable. These strengthen the results in [Discrete Appl. Math. 157~(2) (2009) 433--436] and [Discrete Math. 343 (2020) 111797].

\end{abstract}

\section{Introduction}

A \emph{proper $k$-coloring} of a graph $G$ is a mapping $\phi: V(G) \rightarrow [k]$ such that $\phi(u) \neq \phi(v)$, whenever $uv\in E(G)$, where and herein after, $[k] = \{1, 2, \dots, k\}$. The least integer $k$ such that $G$ admits a proper $k$-coloring is the \emph{chromatic number} $\chi(G)$ of $G$. Let $h$ be a non-negative integer. An \emph{$h$-defective $k$-coloring} of $G$ is a mapping $\phi: V(G)\rightarrow [k]$ such that each color class induces a subgraph of maximum degree at most $h$. In particular, a $0$-defective coloring is a proper coloring of $G$.

A \emph{$k$-list assignment} of $G$ is a mapping $L$ that assigns a list $L(v)$ of $k$ colors to each vertex $v$ in $G$. An \emph{$h$-defective $L$-coloring} of $G$ is an $h$-defective coloring $\psi$ of $G$ such that $\psi(v)\in L(v)$ for all $v\in V(G)$. A graph $G$ is \emph{$h$-defective $k$-choosable} if $G$ admits an $h$-defective $L$-coloring for each $k$-list assignment $L$. In particular, if $G$ is $0$-defective $k$-choosable, then we call it \emph{$k$-choosable}. The \emph{choice number} $ch(G)$ is the smallest integer $k$ such that $G$ is $k$-choosable.

Cowen, Cowen, and Woodall~\cite{MR890224} proved that every outerplanar graph is $2$-defective $2$-colorable, and every planar graph is $2$-defective $3$-colorable. Eaton and Hull~\cite{MR1668108}, and independently, \v{S}krekovski~\cite{MR1702609} proved that every outerplanar graph is $2$-defective $2$-choosable, and every planar graph is $2$-defective $3$-choosable. Cushing and Kierstead~\cite{MR2644426} proved that every planar graph is $1$-defective $4$-choosable. 
Let $\mathcal{G}_{4, l}$ be the family of planar graphs which contain no $4$-cycles and no $l$-cycles. Lih \etal~\cite{MR1820611} proved that for each $l \in \{5, 6, 7\}$, every graph $G \in \mathcal{G}_{4, l}$ is $1$-defective $3$-choosable. Dong and Xu~\cite{MR2479819} proved that for each $l \in \{8, 9\}$, every graph $G \in \mathcal{G}_{4, l}$ is $1$-defective $3$-choosable.

Note that a graph being $h$-defective $k$-choosable means that for every $k$-list assignment $L$ of $G$, there exists a subgraph $H$ (depending on $L$) of $G$ with $\Delta(H) \le h$ such that $G-E(H)$ is $L$-colorable. The subgraph $H$ may be different for different $L$. 
As a strengthening of the above results, the following problem is studied in the literature: For $(h,k) \in \{(2,3), (1,4)\}$, is it true that every planar graph $G$ has a subgraph of maximum degree $h$ such that $G-E(H)$ is $k$-choosable? For $l \in \{5,6,7,8,9\}$, is it true that every graph $G \in \mathcal{G}_{4, l}$ has a matching $M$ such that $G-M$ is $3$-choosable?

It turns out that for the first question, the answer is negative for $(h,k)=(2,3)$, and positive for $(h,k)=(1,4)$. It was proved in~\cite{Kim2019a} that there exists a planar graph $G$ such that for any subgraph $H$ of $G$ of maximum degree $3$, $G - E(H)$ is not $3$-choosable, and proved in~\cite{MR4152773} that every planar graph $G$ has a matching $M$ such that $G-M$ is $4$-choosable. For the second question, for $l \in \{5,6,7\}$, it was shown in~\cite{MR4051856} every graph $G \in \mathcal{G}_{4, l}$ has a matching $M$ such that $G-M$ is $3$-choosable. 

Indeed, stronger results were proved in~\cite{MR4152773,MR4051856}. The results concern two other graph parameters: The \emph{Alon-Tarsi number} $AT(G)$ of $G$ and the \emph{paint number} $\chi_P(G)$ of $G$. The reader is referred to~\cite{MR4152773} for the definitions. We just note here that for any graph $G$, $ch(G) \le \chi_P(G) \le AT(G)$, and the differences $\chi_P(G)-ch(G)$ and $AT(G) - \chi_P(G)$ can be arbitrarily large. It was proved in~\cite{MR4152773} that every planar graph $G$ has a matching $M$ such that $AT(G-M) \le 4$, and proved in~\cite{MR4051856} that for $l \in \{5,6,7\}$, every graph $G \in \mathcal{G}_{4, l}$ has a matching $M$ such that $AT(G-M) \le 3$. 

In this paper, we consider further strengthening of the results concerning graphs in $\mathcal{G}_{4, l}$ for $l \in \{5,6,7,8,9\}$. (Note that the result in~\cite{MR4051856} does not cover the cases for $l=8$ and $9$). We strengthen the above results in two aspects: a larger class of graphs with a stronger property.

Given two non-negative integers $d, h$ and a graph $G$, a 
\emph{$(d,h)$-decomposition} of $G$ is a pair $(D, H)$ such that $H$ is a subgraph of $G$ of maximum degree at most $h$ and $D$ is an acyclic orientation of $G-E(H)$ with maximum out-degree at most $d$. We say $G$ is \emph{$(d,h)$-decomposable} if $G$ has a $(d,h)$-decomposition. Cho \etal~\cite{MR4472765} proved that every planar graph is $(4, 1)$-decomposable, $(3, 2)$-decomposable and $(2, 6)$-decomposable. Note that a graph $H$ which has an acyclic orientation of maximum out-degree at most $d$ if and only if $H$ is \emph{$d$-degenerate}, i.e., the vertices of $H$ can be linearly ordered so that each vertex has at most $d$ backward neighbors.
It is well-known and easy to see that $d$-degenerate graphs not only have choice number, paint number, Alon-Tarsi number and DP-chromatic number at most $d+1$, there is a linear time algorithm that creates the above mentioned linear ordering and the corresponding coloring is easily obtained by using a greedy coloring algorithm. 
 The reader is referred to~\cite{MR3758240} for the definition of DP-chromatic number $\chi_{DP}(G)$ of a graph $G$. We just mention here that $ch(G) \le \chi_{DP}(G)$, and there are graphs $G$ for which $\chi_{DP}(G)$ are larger than each of $AT(G)$ and $\chi_P(G)$, there are also graphs $G$ for which $\chi_{DP}(G)$ are smaller than each of $AT(G)$ and $\chi_P(G)$~\cite{MR4117373}. This paper proves the following result:
 
  \begin{theorem}\label{thm-main}
  	Assume $G$ is a plane graph. Then $G$ is $(2,1)$-decomposable if one of the following holds:
  	\begin{enumerate}
  		\item[(1)] $G$ has no subgraph isomorphic to any configuration in \autoref{COMMONFIGURE} and \autoref{FIGURE-AT567}.
  		\item[(2)] $G$ has no subgraph isomorphic to any configuration in \autoref{COMMONFIGURE} and \autoref{FIGURE-AT48}.
  		\item[(3)] $G \in \mathcal{G}_{4, 9}$.
  	\end{enumerate}
  \end{theorem}

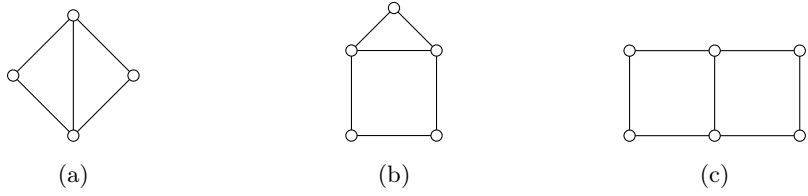
\begin{figure}
\centering
\subcaptionbox{\label{CommonA}}[0.25\linewidth]{\begin{tikzpicture}[scale=0.8]
\def\s{1}
\coordinate (E) at (\s, 0);
\coordinate (N) at (0,\s);
\coordinate (W) at (-\s,0);
\coordinate (S) at (0,-\s);
\draw (E)--(N)--(W)--(S)--cycle;
\draw (N)--(S);
\node[circle, inner sep = 1.5, fill = white, draw] () at (E) {};
\node[circle, inner sep = 1.5, fill = white, draw] () at (N) {};
\node[circle, inner sep = 1.5, fill = white, draw] () at (W) {};
\node[circle, inner sep = 1.5, fill = white, draw] () at (S) {};
\end{tikzpicture}}
\subcaptionbox{\label{CommonB}}[0.25\linewidth]{\begin{tikzpicture}[scale=0.8]
\def\s{1}
\coordinate (NE) at (45:\s);
\coordinate (NW) at (135:\s);
\coordinate (SW) at (225:\s);
\coordinate (SE) at (-45:\s);
\coordinate (H) at (90:1.414*\s);
\draw (NE)--(H)--(NW)--(SW)--(SE)--cycle;
\draw (NE)--(NW);
\node[circle, inner sep = 1.5, fill = white, draw] () at (NE) {};
\node[circle, inner sep = 1.5, fill = white, draw] () at (NW) {};
\node[circle, inner sep = 1.5, fill = white, draw] () at (SW) {};
\node[circle, inner sep = 1.5, fill = white, draw] () at (SE) {};
\node[circle, inner sep = 1.5, fill = white, draw] () at (H) {};
\end{tikzpicture}}
\subcaptionbox{\label{CommonC}}[0.25\linewidth]{\begin{tikzpicture}[scale=0.8]
\def\s{1.414}
\coordinate (NE) at (\s, \s);
\coordinate (N) at (0,\s);
\coordinate (NW) at (-\s,\s);
\coordinate (SW) at (-\s, 0);
\coordinate (S) at (0, 0);
\coordinate (SE) at (\s, 0);
\draw (NE)--(NW)--(SW)--(SE)--cycle;
\draw (N)--(S);
\node[circle, inner sep = 1.5, fill = white, draw] () at (NE) {};
\node[circle, inner sep = 1.5, fill = white, draw] () at (N) {};
\node[circle, inner sep = 1.5, fill = white, draw] () at (NW) {};
\node[circle, inner sep = 1.5, fill = white, draw] () at (SW) {};
\node[circle, inner sep = 1.5, fill = white, draw] () at (S) {};
\node[circle, inner sep = 1.5, fill = white, draw] () at (SE) {};
\end{tikzpicture}}
\caption{Forbidden configurations in (1) and (2) of \autoref{thm-main}.}
\label{COMMONFIGURE}
\end{figure}

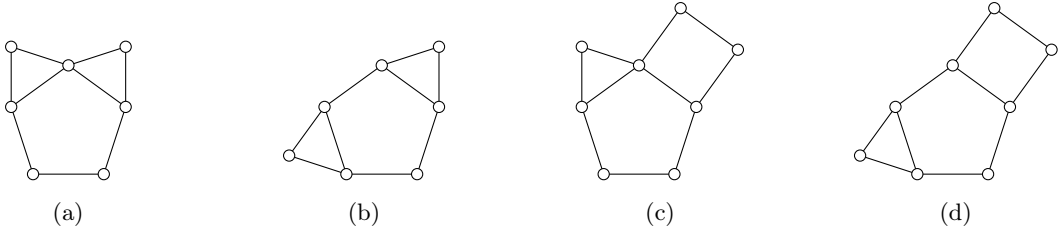
\begin{figure}
\centering
\subcaptionbox{\label{AT345A}}[0.23\linewidth]{\begin{tikzpicture}[scale=0.8]
\def\s{1}
\foreach \ang in {1, 2, 3, 4, 5}
{
\def\pointname{v\ang}
\coordinate (\pointname) at ($(\ang*360/5-54:\s)$);
}
\coordinate (NE) at ($(v1)+(v2)$);
\coordinate (NW) at ($(v2)+(v3)$);
\draw (v1)--(v2)--(v3)--(v4)--(v5)--cycle;
\draw (v1)--(NE)--(v2);
\draw (v2)--(NW)--(v3);
\foreach \ang in {1, 2, 3, 4, 5}
{
\node[circle, inner sep = 1.5, fill = white, draw] () at (v\ang) {};
}
\node[circle, inner sep = 1.5, fill = white, draw] () at (NE) {};
\node[circle, inner sep = 1.5, fill = white, draw] () at (NW) {};
\end{tikzpicture}}
\subcaptionbox{\label{AT345B}}[0.23\linewidth]{\begin{tikzpicture}[scale=0.8]
\def\s{1}
\foreach \ang in {1, 2, 3, 4, 5}
{
\def\pointname{v\ang}
\coordinate (\pointname) at ($(\ang*360/5-54:\s)$);
}
\coordinate (NE) at ($(v1)+(v2)$);
\coordinate (SW) at ($(v3)+(v4)$);
\draw (v1)--(v2)--(v3)--(v4)--(v5)--cycle;
\draw (v1)--(NE)--(v2);
\draw (v3)--(SW)--(v4);
\foreach \ang in {1, 2, 3, 4, 5}
{
\node[circle, inner sep = 1.5, fill = white, draw] () at (v\ang) {};
}
\node[circle, inner sep = 1.5, fill = white, draw] () at (NE) {};
\node[circle, inner sep = 1.5, fill = white, draw] () at (SW) {};
\end{tikzpicture}}
\subcaptionbox{\label{AT345C}}[0.23\linewidth]{\begin{tikzpicture}[scale=0.8]
\def\s{1}
\foreach \ang in {1, 2, 3, 4, 5}
{
\def\pointname{v\ang}
\coordinate (\pointname) at ($(\ang*360/5-54:\s)$);
}
\coordinate (NE2) at ($(v2)!1!90:(v1)$);
\coordinate (NE1) at ($(v1)!1!-90:(v2)$);
\coordinate (NW) at ($(v2)+(v3)$);
\draw (v1)--(v2)--(v3)--(v4)--(v5)--cycle;
\draw (v2)--(NW)--(v3);
\draw (v2)--(NE2)--(NE1)--(v1);
\foreach \ang in {1, 2, 3, 4, 5}
{
\node[circle, inner sep = 1.5, fill = white, draw] () at (v\ang) {};
}
\node[circle, inner sep = 1.5, fill = white, draw] () at (NE1) {};
\node[circle, inner sep = 1.5, fill = white, draw] () at (NE2) {};
\node[circle, inner sep = 1.5, fill = white, draw] () at (NW) {};
\end{tikzpicture}}
\subcaptionbox{\label{AT345D}}[0.23\linewidth]{\begin{tikzpicture}[scale=0.8]
\def\s{1}
\foreach \ang in {1, 2, 3, 4, 5}
{
\def\pointname{v\ang}
\coordinate (\pointname) at ($(\ang*360/5-54:\s)$);
}
\coordinate (NE2) at ($(v2)!1!90:(v1)$);
\coordinate (NE1) at ($(v1)!1!-90:(v2)$);
\coordinate (SW) at ($(v3)+(v4)$);
\draw (v1)--(v2)--(v3)--(v4)--(v5)--cycle;
\draw (v3)--(SW)--(v4);
\draw (v2)--(NE2)--(NE1)--(v1);
\foreach \ang in {1, 2, 3, 4, 5}
{
\node[circle, inner sep = 1.5, fill = white, draw] () at (v\ang) {};
}
\node[circle, inner sep = 1.5, fill = white, draw] () at (NE1) {};
\node[circle, inner sep = 1.5, fill = white, draw] () at (NE2) {};
\node[circle, inner sep = 1.5, fill = white, draw] () at (SW) {};
\end{tikzpicture}}
\caption{Forbidden configurations in (1) of \autoref{thm-main}.}
\label{FIGURE-AT567}
\end{figure}

\begin{figure}
\centering
\subcaptionbox{\label{AT48A}}[0.19\linewidth]{\begin{tikzpicture}[scale=0.8]
\def\s{1.414}
\coordinate (A) at (\s, \s);
\coordinate (B) at (0,\s);
\coordinate (C) at (-\s,\s);
\coordinate (D) at (-\s, 0);
\coordinate (E) at (0, 0);
\coordinate (F) at (\s, 0);
\coordinate (H1) at (0.5*\s, 1.5*\s);
\coordinate (H2) at (-0.5*\s, 1.5*\s);
\draw (A)--(H1)--(B)--(H2)--(C)--(D)--(E)--(E)--(F)--cycle;
\draw (B)--(E);
\node[circle, inner sep = 1.5, fill = white, draw] () at (A) {};
\node[circle, inner sep = 1.5, fill = white, draw] () at (B) {};
\node[circle, inner sep = 1.5, fill = white, draw] () at (C) {};
\node[circle, inner sep = 1.5, fill = white, draw] () at (D) {};
\node[circle, inner sep = 1.5, fill = white, draw] () at (E) {};
\node[circle, inner sep = 1.5, fill = white, draw] () at (F) {};
\node[circle, inner sep = 1.5, fill = white, draw] () at (H1) {};
\node[circle, inner sep = 1.5, fill = white, draw] () at (H2) {};
\end{tikzpicture}}
\subcaptionbox{\label{AT48B}}[0.19\linewidth]{\begin{tikzpicture}[scale=0.8]
\def\s{1.414}
\coordinate (NE) at (0.5*\s, 0.5*\s);
\coordinate (NW) at (-0.5*\s, 0.5*\s);
\coordinate (SW) at (-0.5*\s, -0.5*\s);
\coordinate (SE) at (0.5*\s, -0.5*\s);
\coordinate (H1) at (0, \s);
\coordinate (H2) at (0, 0.6*\s);
\draw (NE)--(H1)--(NW)--(SW)--(SE)--cycle;
\draw (NE)--(H2)--(NW);
\node[circle, inner sep = 1.5, fill = white, draw] () at (NE) {};
\node[circle, inner sep = 1.5, fill = white, draw] () at (NW) {};
\node[circle, inner sep = 1.5, fill = white, draw] () at (SE) {};
\node[circle, inner sep = 1.5, fill = white, draw] () at (SW) {};
\node[circle, inner sep = 1.5, fill = white, draw] () at (H1) {};
\node[circle, inner sep = 1.5, fill = white, draw] () at (H2) {};
\end{tikzpicture}}
\subcaptionbox{\label{AT48C}}[0.19\linewidth]{\begin{tikzpicture}[scale=0.8]
\def\s{1}
\foreach \ang in {1, 2, 3, 4, 5, 6, 7}
{
\def\pointname{v\ang}
\coordinate (\pointname) at ($(\ang*360/7-90:\s)$);
}
\coordinate (W) at ($(v3)!1!-60:(v4)$);
\draw (v1)--(v2)--(v3)--(v4)--(v5)--(v6)--(v7)--cycle;
\draw (v3)--(W)--(v4);
\foreach \ang in {1, 2, 3, 4, 5, 6, 7}
{
\node[circle, inner sep = 1.5, fill = white, draw] () at (v\ang) {};
}
\node[circle, inner sep = 1.5, fill = white, draw] () at (W) {};
\end{tikzpicture}}
\subcaptionbox{\label{AT48D}}[0.19\linewidth]{\begin{tikzpicture}[scale=0.8]
\def\s{1}
\foreach \ang in {1, 2, 3, 4, 5, 6}
{
\def\pointname{v\ang}
\coordinate (\pointname) at ($(\ang*360/6:\s)$);
}
\coordinate (N) at ($(v1)+(v2)$);
\coordinate (SW1) at ($(v3)!1!-90:(v4)$);
\coordinate (SW2) at ($(v4)!1!90:(v3)$);
\draw (v1)--(v2)--(v3)--(v4)--(v5)--(v6)--cycle;
\draw (v1)--(N)--(v2);
\draw (v3)--(SW1)--(SW2)--(v4);
\foreach \ang in {1, 2, 3, 4, 5, 6}
{
\node[circle, inner sep = 1.5, fill = white, draw] () at (v\ang) {};
}
\node[circle, inner sep = 1.5, fill = white, draw] () at (N) {};
\node[circle, inner sep = 1.5, fill = white, draw] () at (SW1) {};
\node[circle, inner sep = 1.5, fill = white, draw] () at (SW2) {};
\end{tikzpicture}}
\subcaptionbox{\label{AT48E}}[0.19\linewidth]{\begin{tikzpicture}[scale=0.8]
\def\s{1}
\foreach \ang in {1, 2, 3, 4, 5, 6}
{
\def\pointname{v\ang}
\coordinate (\pointname) at ($(\ang*360/6:\s)$);
}
\coordinate (W) at ($(v1)+(v2)$);
\coordinate (S1) at ($(v4)!1!-90:(v5)$);
\coordinate (S2) at ($(v5)!1!90:(v4)$);
\draw (v1)--(v2)--(v3)--(v4)--(v5)--(v6)--cycle;
\draw (v1)--(W)--(v2);
\draw (v4)--(S1)--(S2)--(v5);
\foreach \ang in {1, 2, 3, 4, 5, 6}
{
\node[circle, inner sep = 1.5, fill = white, draw] () at (v\ang) {};
}
\node[circle, inner sep = 1.5, fill = white, draw] () at (W) {};
\node[circle, inner sep = 1.5, fill = white, draw] () at (S1) {};
\node[circle, inner sep = 1.5, fill = white, draw] () at (S2) {};
\end{tikzpicture}}
\caption{Forbidden configurations in (2) of \autoref{thm-main}.}
\label{FIGURE-AT48}
\end{figure}

Note that if $G \in \mathcal{G}_{4,l}$ for some $l \in \{5,6,7\}$, then $G$ has no subgraph isomorphic to any configuration in \autoref{COMMONFIGURE} and \autoref{FIGURE-AT567}, and if $G \in \mathcal{G}_{4,8}$, then $G$ has no subgraph isomorphic to any configuration in \autoref{COMMONFIGURE} and \autoref{FIGURE-AT48}. Consequently, for $l \in \{5,6,7,8,9\}$, all graphs $G \in \mathcal{G}_{4,l}$ are $(2,1)$-decomposable.

 
All graphs in this paper are finite and simple. For a plane graph $G$, we use $V(G)$, $E(G)$ and $F(G)$ to denote the vertex set, edge set and face set of $G$, respectively. For any element $x \in V(G)\cup F(G)$, the degree of $x$ is denoted by $d(x)$. A vertex $v$ in $G$ is called a $k$-vertex, or $k^{+}$-vertex, or $k^{-}$-vertex, if $d(v) = k$, or $d(v) \geq k$, or $d(v) \leq k$, respectively. Analogously, one can define $k$-face, $k^{+}$-face, and $k^{-}$-face. An $n$-face $[x_{1}x_{2}\dots x_{n}]$ is a $(d_{1}, d_{2}, \dots, d_{n})$-face if $d(x_{i}) = d_{i}$ for $1 \leq i \leq n$. Let $D$ be an orientation of a graph $G$, we use $d_{D}^{+}(v)$ and $d_{D}^{-}(v)$ to denote the out-degree and in-degree of a vertex $v$ in $D$, respectively. Let $\Delta^{+}(D)$ denote the maximum out-degree of vertices in $D$. Two cycles (or faces) are \emph{adjacent} if they have at least one common edge. Two cycles (or faces) are \emph{normally adjacent} if they intersect in exactly two vertices. Let $G$ be a plane graph and $xy$ be a given boundary edge of $G$. A vertex $v \ne x, y$ is called a \emph{normal vertex}. A vertex $v$ is \emph{special} if $v$ is a $5^{+}$-vertex or $v\in \{x, y\}$. A face is \emph{internal} if it is not the outer face $f_{0}$. A face is \emph{special} if it is an internal $7^{+}$-face or the outer face $f_{0}$. A normal vertex $v$ is \emph{minor} if $d(v)=3$ and it is incident with an internal $4^{-}$-face. A \emph{good $5$-face} is an internal $5$-face adjacent to at least one internal $3$-face. An edge contained in a triangle is a \emph{triangular edge}. Note that in all three cases, there are no adjacent triangles. So every triangular edge is contained in a unique triangle.

\section{Proof of \autoref{thm-main}}

For the purpose of using induction, we prove the following result. Assume $G$ is a plane graph and $e=xy$ is a boundary edge of $G$. A \emph{nice decomposition} of $(G,e)$ is a pair $(D,M)$ such that $M$ is a matching and $D$ is an acyclic orientation of $G - M$ with $d_{D}^{+}(x) = d_{D}^{+}(y) = 0$ and $\Delta^{+}(D) \le 2$. Note that in a nice decomposition $(D,M)$ of $(G,e)$, since $d_{D}^{+}(x) = d_{D}^{+}(y) = 0$, we conclude that $e=xy \in M$. 

\begin{theorem}\label{HP}
If $G$ is a plane graph satisfying the condition of \autoref{thm-main} and $e$ is a boundary edge of $G$, then $(G,e)$ has a nice decomposition. 
\end{theorem}

Assume \autoref{HP} is not true and $G$ is a counterexample with minimum number of vertices. We shall derive a sequence of properties of $G$ that lead to a contradiction. It is obvious that $G$ is connected, for otherwise we can consider each component of $G$ separately.

\begin{lemma}\label{TWOCONNECTED}
$G$ is $2$-connected. 
\end{lemma}
\begin{proof}
Assume to the contrary that $G$ has a cut-vertex $x'$. Let $G = H_{1} \cup H_{2}$, $V(H_{1} \cap H_{2}) = \{x'\}$ and $e=xy \in E(H_{1})$. Let $e'=x'y'$ be a boundary edge of $H_{2}$. By the minimality of $G$, there is a nice decomposition $(D_1,M_1)$ of $(H_1, e)$ and a nice decomposition $(D_2, M_2)$ of $(H_2, e')$. Let $M = (M_{1} \cup M_{2}) \setminus \{x'y'\}$ and $D = D_{1} \cup D_{2} \cup \{\overleftarrow{x'y'}\}$. It is straightforward to verify that $(D, M)$ is a nice decomposition of $(G, e)$. 
\end{proof}

\begin{lemma}\label{delta}
For any $v \in V(G)\setminus \{x, y\}$, $d(v) \geq 3$. 
\end{lemma}
\begin{proof}
Assume $v \in V(G)\setminus \{x, y\}$ and $d(v) \leq 2$. By the minimality of $G$, there exists a nice decomposition $(D,M)$ of $(G-v, e)$. Let $D'$ be obtained from $D$ by orienting edges incident with $v$ as out-going edges from $v$. Then $(D',M)$ is a nice decomposition of $(G,e)$. 
\end{proof}

\begin{lemma}\label{a3}
If $u$ and $v$ are two adjacent $3$-vertices, then $\{u, v\} \cap \{x, y\} \neq \emptyset$. 
\end{lemma}
\begin{proof}
Suppose that $u$ and $v$ are two adjacent $3$-vertices with $\{u, v\} \cap \{x, y\} = \emptyset$. By the minimality of $G$, there is a nice decomposition $(D,M)$ of $(G-\{u,v\}, e)$. Let $M' = M \cup \{uv\}$, and $D'$ be obtained from $D$ by orienting the other edges incident with $u,v$ as out-going edges from $u,v$. Then $(D',M')$ is a nice decomposition of $(G,e)$. 
\end{proof}

For an internal face $f$, let $t_f$ be the number of incident normal $3$-vertices and let $s_f$ be the number of adjacent internal $3$-faces. Note that each $3$-vertex of $f$ is incident with at most one $3$-face adjacent to $f$. Thus we have the following corollary.

\begin{corollary}
	 \label{cor-1}
	 For any internal face $f$, $t_f \le d(f)/2$ and $t_f+s_f \le d(f)$. 
\end{corollary} 

The following four lemmas first appeared in~\cite{MR4051856}, although the hypotheses and some definitions are slightly different. For the completeness of this paper, we include the short proofs with illustration figures. 

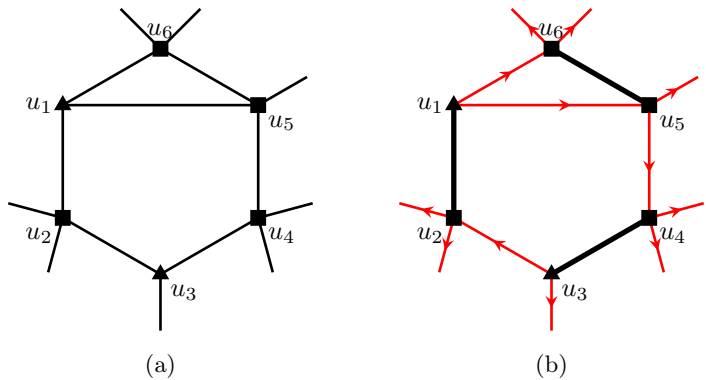
\begin{figure}[htbp]%
\centering
\subcaptionbox{\label{fig4a}}{\begin{tikzpicture}[line width = 1pt]
\def\s{1.5};
\coordinate (u1) at (150:\s);
\coordinate (u2) at (210:\s);
\coordinate (u3) at (270:\s);
\coordinate (u4) at (330:\s);
\coordinate (u5) at (30:\s);
\coordinate (u6) at (90:\s);
\draw (u1)node[left]{$u_{1}$}--(u2)node[below left]{$u_{2}$}--(u3)node[below right]{$u_{3}$}--(u4)node[below right]{$u_{4}$}--(u5)node[below right]{$u_{5}$}--(u6)node[above]{$u_{6}$}--cycle;
\draw (u1)--(u5);
\draw (u6)--($(u6)+(135:0.5*\s)$);
\draw (u6)--($(u6)+(45:0.5*\s)$);
\draw (u5)--($(u5)+(30:0.5*\s)$);
\draw (u4)--($(u4)+(15:0.5*\s)$);
\draw (u4)--($(u4)+(-75:0.5*\s)$);
\draw (u3)--($(u3)+(0, -0.5*\s)$);
\draw (u2)--($(u2)+(165:0.5*\s)$);
\draw (u2)--($(u2)+(255:0.5*\s)$);
\node[regular polygon, regular polygon sides=3, inner sep = 1, fill, draw] () at (u1) {};
\node[rectangle, inner sep = 2.5, fill, draw] () at (u2) {};
\node[regular polygon, regular polygon sides=3, inner sep = 1, fill, draw] () at (u3) {};
\node[rectangle, inner sep = 2.5, fill, draw] () at (u4) {};
\node[rectangle, inner sep = 2.5, fill, draw] () at (u5) {};
\node[rectangle, inner sep = 2.5, fill, draw] () at (u6) {};
\end{tikzpicture}}\hspace{1cm}
\subcaptionbox{\label{fig4b}}{\begin{tikzpicture}[line width = 1pt]
\def\s{1.5};
\coordinate (u1) at (150:\s);
\coordinate (u2) at (210:\s);
\coordinate (u3) at (270:\s);
\coordinate (u4) at (330:\s);
\coordinate (u5) at (30:\s);
\coordinate (u6) at (90:\s);

\draw[line width = 2pt] 
(u1)--(u2) 
(u3)--(u4) 
(u5)--(u6)
;

\path [draw=red, postaction={on each segment={mid arrow=red}}]
(u1)node[left]{$u_{1}$}--(u5)
(u1)--(u6)node[above]{$u_{6}$}
(u5)node[below right]{$u_{5}$}--(u4)node[below right]{$u_{4}$}
(u3)node[below right]{$u_{3}$}--(u2)node[below left]{$u_{2}$}
(u6)--($(u6)+(135:0.5*\s)$)
(u6)--($(u6)+(45:0.5*\s)$)
(u5)--($(u5)+(30:0.5*\s)$)
(u4)--($(u4)+(15:0.5*\s)$)
(u4)--($(u4)+(-75:0.5*\s)$)
(u3)--($(u3)+(0, -0.5*\s)$)
(u2)--($(u2)+(165:0.5*\s)$)
(u2)--($(u2)+(255:0.5*\s)$)
;
\node[regular polygon, regular polygon sides=3, inner sep = 1, fill, draw] () at (u1) {};
\node[rectangle, inner sep = 2.5, fill, draw] () at (u2) {};
\node[regular polygon, regular polygon sides=3, inner sep = 1, fill, draw] () at (u3) {};
\node[rectangle, inner sep = 2.5, fill, draw] () at (u4) {};
\node[rectangle, inner sep = 2.5, fill, draw] () at (u5) {};
\node[rectangle, inner sep = 2.5, fill, draw] () at (u6) {};
\end{tikzpicture}}
\caption{(a) A bad 5-cycle and an adjacent triangle. (b) For the proof of \autoref{S}. Here and in figures below,  a solid triangle represents a 3-vertex, a solid square represents a 4-vertex,  a thick line represents an edge in the matching $M$.}
\end{figure}

A $5$-cycle $[u_{1}u_{2}u_{3}u_{4}u_{5}]$ is a \emph{bad $5$-cycle} if it is adjacent to a triangle $[u_{1}u_{5}u_{6}]$ with $u_{i} \notin \{x, y\}$, where $1 \leq i \leq 6$, and $d(u_{1}) = d(u_{3}) = 3$, and $d(u_{2}) = d(u_{4}) = d(u_{5}) = d(u_{6}) = 4$, as depicted in \autoref{fig4a}. 

\begin{lemma}[Lemma 5.2 in~\cite{MR4051856}]\label{S}
There are no bad $5$-cycles in $G$. 
\end{lemma}

\begin{proof}[Proof of \autoref{S}]
Assume $C = [u_{1}u_{2}u_{3}u_{4}u_{5}]$ is a bad $5$-cycle and $T = [u_{1}u_{5}u_{6}]$ is a  triangle adjacent to $C$, where $d(u_{1}) = d(u_{3}) = 3$ and $d(u_{i}) = 4$ for $i \in \{2, 4, 5, 6\}$, as depicted in \autoref{fig4a}. 
A nice decomposition of $G - \{u_{1}, u_{2}, \dots, u_{6}\}$ is extended to a nice decomposition as in \autoref{fig4b}. 
\end{proof}

A triangle $T$ is \emph{minor} if $T$ is a $(3, 4, 4)$-triangle and $T \cap \{x, y\} = \emptyset$. A \emph{triangle chain} in $G$ is a subgraph of $G - \{x, y\}$ consisting of vertices $w_{1}, w_{2}, \dots, w_{k+1}, u_{1}, u_{2}, \dots, u_{k}$ in which $[w_{i}w_{i+1}u_{i}]$ is a $(4, 4, 4)$-cycle for $1 \leq i \leq k$, as depicted in \autoref{TRIANGLE-CHAIN}. We denote $T_{i}$ the triangle $[w_{i}w_{i+1}u_{i}]$ and denote such a triangle chain by $T_{1}T_{2}\dots T_{k}$. If a triangle $T$ has exactly one common vertex with a triangle chain $T_{1}T_{2}\dots T_{k}$ and the common vertex is in $T_{1}$, then we say $T$ \emph{intersects} the triangle chain $T_{1}T_{2}\dots T_{k}$. 

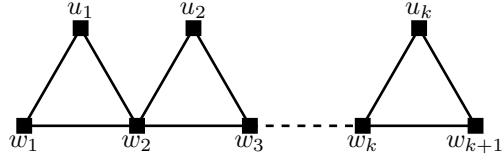
\begin{figure}[htbp]%
\centering
\begin{tikzpicture}[line width = 1pt]
\def\s{1.5};
\coordinate (W1) at (0, 0);
\coordinate (W2) at (\s, 0);
\coordinate (U1) at ($(W1)!1!60:(W2)$);
\draw (U1)node[above]{$u_{1}$}--(W1)node[below]{$w_{1}$}--(W2)node[below]{$w_{2}$}--cycle;
\coordinate (U2) at ($(U1) + (\s, 0)$);
\coordinate (W3) at ($(W2) + (\s, 0)$);
\draw (U2)node[above]{$u_{2}$}--(W2)--(W3)node[below]{$w_{3}$}--cycle;
\coordinate (U4) at ($(U2) + (2*\s, 0)$);
\coordinate (W4) at ($(W2) + (2*\s, 0)$);
\coordinate (W5) at ($(W3) + (2*\s, 0)$);
\draw[dashed] (W3)--(W4);
\draw (U4)node[above]{$u_{k}$}--(W4)node[below]{$w_{k}$}--(W5)node[below]{$w_{k+1}$}--cycle;
\node[rectangle, inner sep = 2.5, fill, draw] () at (U1) {};
\node[rectangle, inner sep = 2.5, fill, draw] () at (U2) {};
\node[rectangle, inner sep = 2.5, fill, draw] () at (U4) {};
\node[rectangle, inner sep = 2.5, fill, draw] () at (W1) {};
\node[rectangle, inner sep = 2.5, fill, draw] () at (W2) {};
\node[rectangle, inner sep = 2.5, fill, draw] () at (W3) {};
\node[rectangle, inner sep = 2.5, fill, draw] () at (W4) {};
\node[rectangle, inner sep = 2.5, fill, draw] () at (W5) {};
\end{tikzpicture}
\caption{A triangle chain.}
\label{TRIANGLE-CHAIN}
\end{figure}

\begin{figure}[htbp]%
\centering
\subcaptionbox{\label{fig6a}}{\begin{tikzpicture}[line width = 1pt]
\def\s{1.5};
\coordinate (W0) at (-\s, 0);
\coordinate (W1) at (0, 0);
\coordinate (U0) at ($(W0)!1!60:(W1)$);
\coordinate (W2) at (\s, 0);
\coordinate (U1) at ($(W1)!1!60:(W2)$);
\coordinate (W3) at (2*\s, 0);
\coordinate (U2) at ($(U1) + (\s, 0)$);
\coordinate (U4) at ($(U2) + (2*\s, 0)$);
\coordinate (W4) at (3*\s, 0);
\coordinate (W5) at (4*\s, 0);
\draw[dashed] (W3)--(W4);
\coordinate (z) at (5*\s, 0);

\draw
(W0)node[below]{$w_{0}$}--(U0)node[right]{$u_{0}$}
(W1)node[below]{$w_{1}$}--(U1)node[right]{$u_{1}$}
(W2)node[below]{$w_{2}$}--(U2)node[right]{$u_{2}$}
(W3)node[below]{$w_{3}$}--($(W3)!0.7!60:(W4)$)
(W4)node[below]{$w_{k}$}--(U4)node[right]{$u_{k}$}
(W5)node[below]{$w_{k+1}$}--(z)node[below]{$z$}
;

\path [draw]
(W0)--($(W0) + (-0.7*\s, 0)$)
(W0)--(W1)--(U0)
(W1)--(W2)--(U1)
(W2)--(W3)--(U2)
(W4)--(W5)--(U4)
(U0)--($(U0)!0.7!-120:(W0)$)
(U0)--($(U0)!0.7!120:(W1)$)
(U1)--($(U1)!0.7!-120:(W1)$)
(U1)--($(U1)!0.7!120:(W2)$)
(U2)--($(U2)!0.7!-120:(W2)$)
(U2)--($(U2)!0.7!120:(W3)$)
(W4)--($(W4)!0.7!-60:(W3)$)
(U4)--($(U4)!0.7!-120:(W4)$)
(U4)--($(U4)!0.7!120:(W5)$)
(W5)--($(W5)!0.7!60:(z)$)
(z)--($(z)!0.7!-120:(W5)$)
(z)--($(z) + (0.7*\s, 0)$)
;

\node[rectangle, inner sep = 2.5, fill, draw] () at (U0) {};
\node[rectangle, inner sep = 2.5, fill, draw] () at (U1) {};
\node[rectangle, inner sep = 2.5, fill, draw] () at (U2) {};
\node[rectangle, inner sep = 2.5, fill, draw] () at (U4) {};
\node[regular polygon, regular polygon sides=3, inner sep = 1, fill, draw] () at (W0) {};
\node[rectangle, inner sep = 2.5, fill, draw] () at (W1) {};
\node[rectangle, inner sep = 2.5, fill, draw] () at (W2) {};
\node[rectangle, inner sep = 2.5, fill, draw] () at (W3) {};
\node[rectangle, inner sep = 2.5, fill, draw] () at (W4) {};
\node[rectangle, inner sep = 2.5, fill, draw] () at (W5) {};
\node[regular polygon, regular polygon sides=3, inner sep = 1, fill, draw] () at (z) {};
\end{tikzpicture}}\vspace{0.5cm}

\subcaptionbox{\label{fig6b}}{\begin{tikzpicture}[line width = 1pt]
\def\s{1.5};
\coordinate (W0) at (-\s, 0);
\coordinate (W1) at (0, 0);
\coordinate (U0) at ($(W0)!1!60:(W1)$);
\coordinate (W2) at (\s, 0);
\coordinate (U1) at ($(W1)!1!60:(W2)$);
\coordinate (W3) at (2*\s, 0);
\coordinate (U2) at ($(U1) + (\s, 0)$);
\coordinate (U4) at ($(U2) + (2*\s, 0)$);
\coordinate (W4) at (3*\s, 0);
\coordinate (W5) at (4*\s, 0);
\draw[dashed] (W3)--(W4);
\coordinate (z) at (5*\s, 0);

\draw[line width = 2pt]
(W0)node[below]{$w_{0}$}--(U0)node[right]{$u_{0}$}
(W1)node[below]{$w_{1}$}--(U1)node[right]{$u_{1}$}
(W2)node[below]{$w_{2}$}--(U2)node[right]{$u_{2}$}
(W3)node[below]{$w_{3}$}--($(W3)!0.7!60:(W4)$)
(W4)node[below]{$w_{k}$}--(U4)node[right]{$u_{k}$}
(W5)node[below]{$w_{k+1}$}--(z)node[below]{$z$}
;

\path [draw=red,postaction={on each segment={mid arrow=red}}]
(W0)--($(W0) + (-0.7*\s, 0)$)
(W0)--(W1)--(U0)
(W1)--(W2)--(U1)
(W2)--(W3)--(U2)
(W4)--(W5)--(U4)
(U0)--($(U0)!0.7!-120:(W0)$)
(U0)--($(U0)!0.7!120:(W1)$)
(U1)--($(U1)!0.7!-120:(W1)$)
(U1)--($(U1)!0.7!120:(W2)$)
(U2)--($(U2)!0.7!-120:(W2)$)
(U2)--($(U2)!0.7!120:(W3)$)
(W4)--($(W4)!0.7!-60:(W3)$)
(U4)--($(U4)!0.7!-120:(W4)$)
(U4)--($(U4)!0.7!120:(W5)$)
(W5)--($(W5)!0.7!60:(z)$)
(z)--($(z)!0.7!-120:(W5)$)
(z)--($(z) + (0.7*\s, 0)$)
;

\node[rectangle, inner sep = 2.5, fill, draw] () at (U0) {};
\node[rectangle, inner sep = 2.5, fill, draw] () at (U1) {};
\node[rectangle, inner sep = 2.5, fill, draw] () at (U2) {};
\node[rectangle, inner sep = 2.5, fill, draw] () at (U4) {};
\node[regular polygon, regular polygon sides=3, inner sep = 1, fill, draw] () at (W0) {};
\node[rectangle, inner sep = 2.5, fill, draw] () at (W1) {};
\node[rectangle, inner sep = 2.5, fill, draw] () at (W2) {};
\node[rectangle, inner sep = 2.5, fill, draw] () at (W3) {};
\node[rectangle, inner sep = 2.5, fill, draw] () at (W4) {};
\node[rectangle, inner sep = 2.5, fill, draw] () at (W5) {};
\node[regular polygon, regular polygon sides=3, inner sep = 1, fill, draw] () at (z) {};
\end{tikzpicture}}
\caption{(a) The configuration in \autoref{TC1}. (b) For the proof of \autoref{TC1}.}
\end{figure}
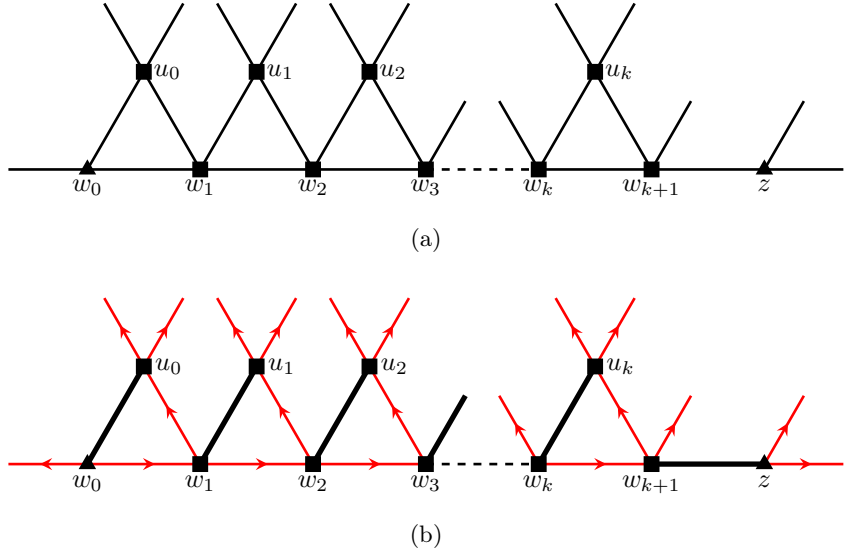

\begin{lemma}[Lemma 2.10 in~\cite{MR4051856}]\label{TC1}
If a minor triangle $T_{0}$ intersects a triangle chain $T_{1}T_{2}\dots T_{k}$, then every $3$-vertex adjacent to a vertex in $T_{k}$ belongs to $\{x, y\} \cup V(T_{0})$. 
\end{lemma}
The $k=0$ case of the above lemma asserts that every 3-vertex adjacent to a vertex in $T_0$ belongs to $\{x, y\}$.

\begin{proof}[Proof of \autoref{TC1}]
Assume $G$ has a minor triangle
$T_{0} = [w_{0}w_{1}u_{0}]$ intersecting a triangle chain $T_{1}T_{2}\dots T_{k}$, and $z \notin \{x, y\} \cup V(T_{0})$ is a 3-vertex adjacent to a vertex in $T_k$,
as depicted in \autoref{fig6a}.
%
A nice decomposition of $G- (\bigcup_{i=0}^{k}V(T_{i}) \cup \{z\})$ is extended to a nice decomposition of $G$ as in \autoref{fig6b}.
%
%
%
\end{proof}

\begin{figure}[htbp]%
\centering
\subcaptionbox{\label{fig7a}}{\begin{tikzpicture}[line width = 1pt]
\def\s{1.5};
\coordinate (W0) at (-\s, 0);
\coordinate (W1) at (0, 0);
\coordinate (U0) at ($(W0)!1!60:(W1)$);
\coordinate (W2) at (\s, 0);
\coordinate (U1) at ($(W1)!1!60:(W2)$);
\coordinate (W3) at (2*\s, 0);
\coordinate (U2) at ($(U1) + (\s, 0)$);
\coordinate (U4) at ($(U2) + (2*\s, 0)$);
\coordinate (W4) at (3*\s, 0);
\coordinate (W5) at (4*\s, 0);
\draw[dashed] (W3)--(W4);
\coordinate (z) at (5*\s, 0);
\coordinate (z1) at (6*\s, 0);
\coordinate (z2) at ($(U4) + (2*\s, 0)$);

\draw
(W0)node[below]{$w_{0}$}--(U0)node[right]{$u_{0}$}
(W1)node[below]{$w_{1}$}--(U1)node[right]{$u_{1}$}
(W2)node[below]{$w_{2}$}--(U2)node[right]{$u_{2}$}
(W3)node[below]{$w_{3}$}--($(W3)!0.7!60:(W4)$)
(W4)node[below]{$w_{k}$}--(U4)node[right]{$u_{k}$}
(W5)node[below]{$w_{k+1}$}--(z)node[below]{$z$}
(z1)node[below]{$z_{1}$}--(z2)node[right]{$z_{2}$}
;

\path [draw]
(W0)--($(W0) + (-0.7*\s, 0)$)
(W0)--(W1)--(U0)
(W1)--(W2)--(U1)
(W2)--(W3)--(U2)
(W4)--(W5)--(U4)
(U0)--($(U0)!0.7!-120:(W0)$)
(U0)--($(U0)!0.7!120:(W1)$)
(U1)--($(U1)!0.7!-120:(W1)$)
(U1)--($(U1)!0.7!120:(W2)$)
(U2)--($(U2)!0.7!-120:(W2)$)
(U2)--($(U2)!0.7!120:(W3)$)
(W4)--($(W4)!0.7!-60:(W3)$)
(U4)--($(U4)!0.7!-120:(W4)$)
(U4)--($(U4)!0.7!120:(W5)$)
(W5)--($(W5)!0.7!60:(z)$)
(z2)--($(z2)!0.7!-120:(z)$)
(z2)--($(z2)!0.7!120:(z1)$)
(z)--($(z)!0.7!-60:(W5)$)
(z)--(z2)
(z1)--(z)
(z1)--($(z1) + (0.7*\s, 0)$)
;

\node[rectangle, inner sep = 2.5, fill, draw] () at (U0) {};
\node[rectangle, inner sep = 2.5, fill, draw] () at (U1) {};
\node[rectangle, inner sep = 2.5, fill, draw] () at (U2) {};
\node[rectangle, inner sep = 2.5, fill, draw] () at (U4) {};
\node[regular polygon, regular polygon sides=3, inner sep = 1, fill, draw] () at (W0) {};
\node[rectangle, inner sep = 2.5, fill, draw] () at (W1) {};
\node[rectangle, inner sep = 2.5, fill, draw] () at (W2) {};
\node[rectangle, inner sep = 2.5, fill, draw] () at (W3) {};
\node[rectangle, inner sep = 2.5, fill, draw] () at (W4) {};
\node[rectangle, inner sep = 2.5, fill, draw] () at (W5) {};
\node[rectangle, inner sep = 2.5, fill, draw] () at (z) {};
\node[regular polygon, regular polygon sides=3, inner sep = 1, fill, draw] () at (z1) {};
\node[rectangle, inner sep = 2.5, fill, draw] () at (z2) {};
\end{tikzpicture}}\vspace{0.5cm}

\subcaptionbox{\label{fig7b}}{\begin{tikzpicture}[line width = 1pt]
\def\s{1.5};
\coordinate (W0) at (-\s, 0);
\coordinate (W1) at (0, 0);
\coordinate (U0) at ($(W0)!1!60:(W1)$);
\coordinate (W2) at (\s, 0);
\coordinate (U1) at ($(W1)!1!60:(W2)$);
\coordinate (W3) at (2*\s, 0);
\coordinate (U2) at ($(U1) + (\s, 0)$);
\coordinate (U4) at ($(U2) + (2*\s, 0)$);
\coordinate (W4) at (3*\s, 0);
\coordinate (W5) at (4*\s, 0);
\draw[dashed] (W3)--(W4);
\coordinate (z) at (5*\s, 0);
\coordinate (z1) at (6*\s, 0);
\coordinate (z2) at ($(U4) + (2*\s, 0)$);

\draw[line width = 2pt]
(W0)node[below]{$w_{0}$}--(U0)node[right]{$u_{0}$}
(W1)node[below]{$w_{1}$}--(U1)node[right]{$u_{1}$}
(W2)node[below]{$w_{2}$}--(U2)node[right]{$u_{2}$}
(W3)node[below]{$w_{3}$}--($(W3)!0.7!60:(W4)$)
(W4)node[below]{$w_{k}$}--(U4)node[right]{$u_{k}$}
(W5)node[below]{$w_{k+1}$}--(z)node[below]{$z$}
(z1)node[below]{$z_{1}$}--(z2)node[right]{$z_{2}$}
;

\path [draw=red,postaction={on each segment={mid arrow=red}}]
(W0)--($(W0) + (-0.7*\s, 0)$)
(W0)--(W1)--(U0)
(W1)--(W2)--(U1)
(W2)--(W3)--(U2)
(W4)--(W5)--(U4)
(U0)--($(U0)!0.7!-120:(W0)$)
(U0)--($(U0)!0.7!120:(W1)$)
(U1)--($(U1)!0.7!-120:(W1)$)
(U1)--($(U1)!0.7!120:(W2)$)
(U2)--($(U2)!0.7!-120:(W2)$)
(U2)--($(U2)!0.7!120:(W3)$)
(W4)--($(W4)!0.7!-60:(W3)$)
(U4)--($(U4)!0.7!-120:(W4)$)
(U4)--($(U4)!0.7!120:(W5)$)
(W5)--($(W5)!0.7!60:(z)$)
(z2)--($(z2)!0.7!-120:(z)$)
(z2)--($(z2)!0.7!120:(z1)$)
(z)--($(z)!0.7!-60:(W5)$)
(z)--(z2)
(z1)--(z)
(z1)--($(z1) + (0.7*\s, 0)$)
;

\node[rectangle, inner sep = 2.5, fill, draw] () at (U0) {};
\node[rectangle, inner sep = 2.5, fill, draw] () at (U1) {};
\node[rectangle, inner sep = 2.5, fill, draw] () at (U2) {};
\node[rectangle, inner sep = 2.5, fill, draw] () at (U4) {};
\node[regular polygon, regular polygon sides=3, inner sep = 1, fill, draw] () at (W0) {};
\node[rectangle, inner sep = 2.5, fill, draw] () at (W1) {};
\node[rectangle, inner sep = 2.5, fill, draw] () at (W2) {};
\node[rectangle, inner sep = 2.5, fill, draw] () at (W3) {};
\node[rectangle, inner sep = 2.5, fill, draw] () at (W4) {};
\node[rectangle, inner sep = 2.5, fill, draw] () at (W5) {};
\node[rectangle, inner sep = 2.5, fill, draw] () at (z) {};
\node[regular polygon, regular polygon sides=3, inner sep = 1, fill, draw] () at (z1) {};
\node[rectangle, inner sep = 2.5, fill, draw] () at (z2) {};
\end{tikzpicture}}
\caption{(a) The configuration in \autoref{TC2}. (b) For the proof of \autoref{TC2}.}
\end{figure}
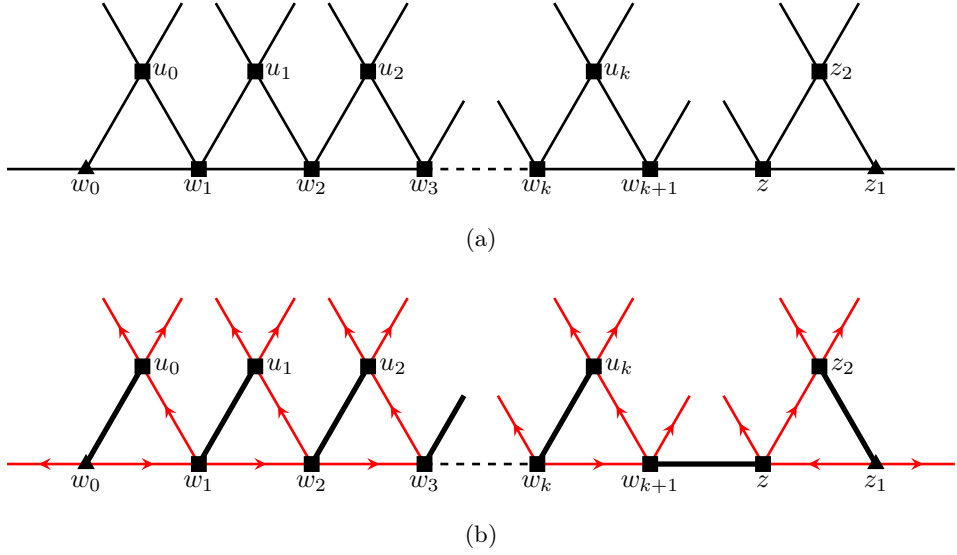

\begin{lemma}[Lemma 2.11 in~\cite{MR4051856}]\label{TC2}
If a minor triangle $T_{0}$ intersects a triangle chain $T_{1}T_{2}\dots T_{k}$, then the distance between $T_{k}$ and another minor triangle is at least two. 
\end{lemma}

\begin{proof}[Proof of \autoref{TC2}]
Assume to the contrary that $T_{1}T_{2} \dots T_{k}$ with  $T_{i} = [w_{i}w_{i+1}u_{i}]$, $1 \le i \le k$, is a triangle chain that intersects a minor triangle $T_{0} = [w_{0}w_{1}u_{0}]$, and the distance between $T_{k}$ and another minor triangle $T_{0}'=[zz_{1}z_{2}]$ with $d(z_{1}) = 3$ is less than 2. By \autoref{TC1}, we may assume $w_{k+1}z$ is a $(4, 4)$-edge connecting $T_{k}$ and $T_{0}'$, as depicted in \autoref{fig7a}.
A nice decomposition of $G-(\bigcup_{i=0}^{k} V(T_{i}) \cup V(T'_{0}))$ is extended to a nice decomposition of $G$ as in  \autoref{fig7b}.
%
%
%
\end{proof}

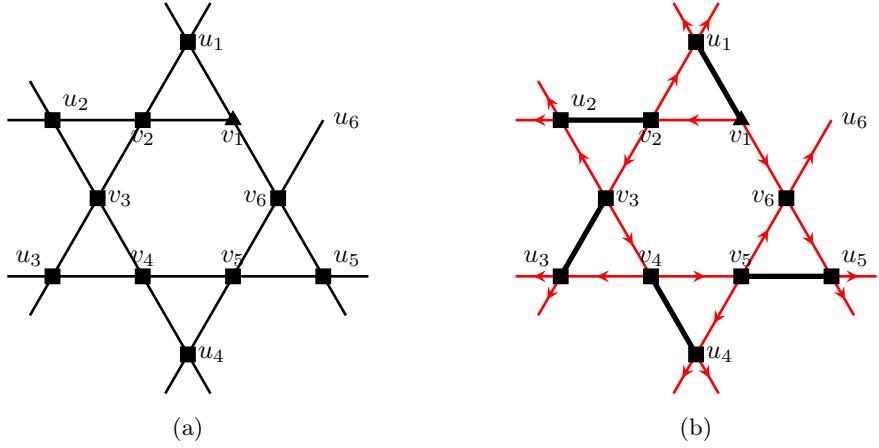
\begin{figure}
\centering
\subcaptionbox{\label{fig8a}}[0.4\linewidth]{\begin{tikzpicture}[line width = 1pt]
\def\s{1.2}
\foreach \ang in {1, 2, 3, 4, 5, 6}
{
\def\pointnamev{v\ang}
\coordinate (\pointnamev) at ($(\ang*360/6:\s)$);
\def\pointnameu{u\ang}
\coordinate (\pointnameu) at ($(\ang*360/6+30:1.732*\s)$);
}
\path [draw]
(v1)node[below]{$v_{1}$}--(v2)node[below]{$v_{2}$}--(v3)node[right]{$v_{3}$}--(v4)node[above]{$v_{4}$}--(v5)node[above]{$v_{5}$}--(v6)node[left]{$v_{6}$}
(v1)--(v6)
(v2)--(u1)
(v3)--(u2)
(v4)--(u3)
(v5)--(u4)
(v6)--(u5)
;
\draw
(v1)--(u1)
(v2)--(u2)
(v3)--(u3)
(v4)--(u4)
(v5)--(u5)
;

\path [draw]
(u1)node[right]{$u_{1}$}--($(u1)+ (60:0.5*\s)$)
(u1)--($(u1)+ (120:0.5*\s)$)
(u2)node[above right]{$u_{2}$}--($(u2)+ (120:0.5*\s)$)
(u2)--($(u2)+ (180:0.5*\s)$)
(u3)node[above left]{$u_{3}$}--($(u3)+ (180:0.5*\s)$)
(u3)--($(u3)+ (240:0.5*\s)$)
(u4)node[right]{$u_{4}$}--($(u4)+ (240:0.5*\s)$)
(u4)--($(u4)+ (300:0.5*\s)$)
(u5)node[above right]{$u_{5}$}--($(u5)+ (300:0.5*\s)$)
(u5)--($(u5)+ (0:0.5*\s)$)
(v6)--(u6)node[right]{$u_{6}$}
;
\foreach \ang in {2, 3, 4, 5, 6}
{
\node[rectangle, inner sep = 2.5, fill, draw] () at (v\ang) {};
}
\foreach \ang in {1, 2, 3, 4, 5}
{
\node[rectangle, inner sep = 2.5, fill, draw] () at (u\ang) {};
}
\node[regular polygon, regular polygon sides=3, inner sep = 1, fill, draw] () at (v1) {};
\end{tikzpicture}}
\subcaptionbox{\label{fig8b}}[0.4\linewidth]{\begin{tikzpicture}[line width = 1pt]
\def\s{1.2}
\foreach \ang in {1, 2, 3, 4, 5, 6}
{
\def\pointnamev{v\ang}
\coordinate (\pointnamev) at ($(\ang*360/6:\s)$);
\def\pointnameu{u\ang}
\coordinate (\pointnameu) at ($(\ang*360/6+30:1.732*\s)$);
}
\path [draw=red,postaction={on each segment={mid arrow=red}}]
(v1)node[below]{$v_{1}$}--(v2)node[below]{$v_{2}$}--(v3)node[right]{$v_{3}$}--(v4)node[above]{$v_{4}$}--(v5)node[above]{$v_{5}$}--(v6)node[left]{$v_{6}$}
(v1)--(v6)
(v2)--(u1)
(v3)--(u2)
(v4)--(u3)
(v5)--(u4)
(v6)--(u5)
;
\draw[line width = 2pt] 
(v1)--(u1)
(v2)--(u2)
(v3)--(u3)
(v4)--(u4)
(v5)--(u5)
;

\path [draw=red,postaction={on each segment={mid arrow=red}}]
(u1)node[right]{$u_{1}$}--($(u1)+ (60:0.5*\s)$)
(u1)--($(u1)+ (120:0.5*\s)$)
(u2)node[above right]{$u_{2}$}--($(u2)+ (120:0.5*\s)$)
(u2)--($(u2)+ (180:0.5*\s)$)
(u3)node[above left]{$u_{3}$}--($(u3)+ (180:0.5*\s)$)
(u3)--($(u3)+ (240:0.5*\s)$)
(u4)node[right]{$u_{4}$}--($(u4)+ (240:0.5*\s)$)
(u4)--($(u4)+ (300:0.5*\s)$)
(u5)node[above right]{$u_{5}$}--($(u5)+ (300:0.5*\s)$)
(u5)--($(u5)+ (0:0.5*\s)$)
(v6)--(u6)node[right]{$u_{6}$}
;
\foreach \ang in {2, 3, 4, 5, 6}
{
\node[rectangle, inner sep = 2.5, fill, draw] () at (v\ang) {};
}
\foreach \ang in {1, 2, 3, 4, 5}
{
\node[rectangle, inner sep = 2.5, fill, draw] () at (u\ang) {};
}
\node[regular polygon, regular polygon sides=3, inner sep = 1, fill, draw] () at (v1) {};
\end{tikzpicture}}
\caption{(a) The configuration in \autoref{TC3}. (b) For the proof of \autoref{TC3}.}
\end{figure}

\begin{lemma}[Lemma 3.1 in~\cite{MR4051856}]\label{TC3}
Assume that $f$ is a $6$-face adjacent to five $3$-faces, and none of the vertices on these $3$-faces is in $\{x, y\}$. If $f$ is incident with a $3$-vertex, then there is at least one $5^{+}$-vertex on these five $3$-faces. 
\end{lemma}

\begin{proof}[Proof of \autoref{TC3}]
Let $f=[v_{1}v_{2}v_{3}v_{4}v_{5}v_{6}]$ be a 6-face, $v_{1}$ be a 3-vertex and $T_{i} = [v_{i}v_{i+1}u_{i}]$, $1 \leq i \leq 5$, be the five $3$-faces. Assume to the contrary that there is no $5^+$-vertex on $T_{i}$. By \autoref{TC1}, we may assume all $v_{i+1}$ and $u_{i}$ are 4-vertices for $1 \leq i \leq 5$, as depicted in \autoref{fig8a}. 
A nice decomposition of $G-(\bigcup_{i=1}^{5} V(T_i))$ is extended to a nice decomposition of $G$ as in  \autoref{fig8b}.
%
%
%
\end{proof}

The above lemmas present some reducible configurations. We use standard discharging method to prove that there must be some reducible configurations in a minimum counterexample,  which leads to a contradiction. 

First, we define an initial charge function by $\mu(x) = d(x) - 4$, $\mu(y) = d(y) - 4$, $\mu(f_{0}) = d(f_{0}) + 4$, and $\mu(v) = d(v) - 4$ for each vertex $v \in V(G) \setminus \{x, y\}$, $\mu(f) = d(f) - 4$ for each face $f$ other than $f_{0}$. By Euler's formula and handshaking theorem, we obtain that the sum of all the initial charges is zero, \ie 
\[
(d(x) - 4) + (d(y) - 4) + (d(f_{0}) + 4) + \sum_{v \neq x, y} (d(v) - 4) + \sum_{f \neq f_{0}} (d(f) - 4) = 0. 
\]
Next, we design some discharging rules to redistribute the charges, such that the sum of the final charges is not zero, which leads to a contradiction.

\paragraph{Discharging Rules}

\begin{enumerate}[label=\textbf{R\arabic*.}, ref=R\arabic*]
\item\label{R1} Every internal $3$-face $f$ receives $\frac{1}{3}$ from each adjacent face. 
\item\label{R2} Assume $v$ is a normal $3$-vertex. If $v$ is incident with an internal $4^-$-face, then it receives $\frac{1}{2}$ from each of the other two incident faces. Otherwise it receives $\frac{1}{3}$ from each incident face. 
\item\label{R3} Let $v$ be a normal $5$-vertex. Then $v$ sends $\frac{1}{6}$ to each incident $4^{+}$-face. If $v$ is incident with a $3$-face $g=[uvw]$, then $v$ sends $\frac{1}{6}$ to the other face $g'$ incident with $uw$. Moreover, if $v$ is incident with three consecutive faces $f_{1}, f_{2}, f_{3}$ and $f_{1}, f_{3}$ are $3$-faces, then $v$ sends an extra $\frac{1}{6}$ to $f_{2}$. 
\item\label{R4} Let $v $ be a normal $6^{+}$-vertex. Then $v$ sends $\frac{1}{3}$ to each incident $4^{+}$-face. If $v$ is incident with a $3$-face $g=[uvw]$, then $v$ sends $\frac{1}{3}$ to the other face $g'$ incident with $uw$. 
\item\label{R5} Let $v$ be a vertex in $\{x, y\}$. Then it sends $\frac{1}{3}$ to every incident internal $4^{+}$-face. If $v$ is incident with a $3$-face $g=[uvw]$, then $v$ sends $\frac{1}{3}$ to the other face $g'$ incident with $uw$. 
\item\label{R6} $f_{0}$ sends $\frac{1}{3}$ to each adjacent $4^{+}$-face. 
\item\label{RR7} In Case 2 (i.e., $G$ has no subgraph isomorphic to any configuration in \autoref{COMMONFIGURE} and \autoref{FIGURE-AT48}), every internal $5$-face receives $\frac{1}{6}$ from adjacent internal $6^{+}$-faces via each common edge.  
\item\label{RRR7} In Case 3 (i.e., $G \in \mathcal{G}_{4,9}$), every good $5$-face receives $\frac{1}{3}$ from adjacent internal $7^{+}$-faces via each common edge.
\end{enumerate}

For $z \in V(G) \cup F(G)$, let $\mu'(z)$ be the final charge of $z$. In the remainder of this paper, we prove that $\sum_{z \in V(G) \cup F(G)} \mu'(z) > 0$, which contradicts the fact that $\sum _{z \in V(G) \cup F(G)} \mu'(z) = \sum_{z \in V(G) \cup F(G)} \mu(z) = 0$.

Note that \ref{RR7} only applies to Case 2 and \ref{RRR7} only applies to Case 3. Moreover, \ref{RR7} and \ref{RRR7} only involve $5^+$-faces.

It follows from \ref{R5} that for $v\in\{x, y\}$
\[
\mu'(v) \geq \mu(v) - (d(v) - 1) \times \frac{1}{3} = \frac{2d(v) - 11}{3} \geq -\frac{7}{3}. 
\]
Note that $f_{0}$ sends $\frac{1}{3}$ to each adjacent internal face by \ref{R1} and \ref{R6}, and sends at most $\frac{1}{2}$ to each incident normal $3$-vertex by \ref{R2}. It follows from \autoref{a3} that $f_{0}$ is incident with at most $\frac{d(f_{0})}{2}$ normal $3$-vertices. Then
\[
\mu'(f_{0}) \geq \mu(f_{0}) - \frac{d(f_{0})}{2} \times \frac{1}{2} - d(f_{0}) \times \frac{1}{3} \geq \frac{5d(f_{0})}{12} + 4 \geq \frac{21}{4}. 
\]

Hence, $\mu'(x) + \mu'(y) + \mu'(f_{0}) > 0$.

Assume $v$ is a normal $3$-vertex. If $v$ is incident with an internal $4^-$-face, then the other two incident faces are $5^{+}$-faces or the outer face $f_{0}$. Hence $\mu'(v) = \mu(v) + 2 \times \frac{1}{2} = 0$. Otherwise each face incident with $v$ is a $5^{+}$-face or $f_{0}$, and $\mu'(v) = \mu(v) + 3 \times \frac{1}{3} = 0$ by \ref{R2}. 

If $v$ is a normal $4$-vertex, then $\mu'(v) = \mu(v) = 0$. If $v$ is a normal $5$-vertex, then it is incident with at most two $3$-faces, and then $\mu'(v) \geq \mu(v) - 5 \times \frac{1}{6} - \frac{1}{6} = 0$ by \ref{R3}. If $v$ is a normal $6^{+}$-vertex, then $\mu'(v) = \mu(v) - d(v) \times \frac{1}{3} = \frac{2(d(v) - 6)}{3} \geq 0$ by \ref{R4}. 

If $f$ is an internal $3$-face, then it receives $\frac{1}{3}$ via each incident edge, and $\mu'(f) = \mu(f) + 3 \times \frac{1}{3} = 0$ by \ref{R1}.

If $f$ is an internal $4$-face, then $\mu'(f) \geq \mu(f) = 0$. 

It remains to show that $\mu'(f) \geq 0$ for internal $5^+$-faces $f$. 

In the remainder of the paper, we consider the three cases separately in three subsections.

\subsection{\texorpdfstring{$G$}{G} has no subgraph isomorphic to any configuration in \autoref{COMMONFIGURE} and \autoref{FIGURE-AT567}}


Assume that $f = [v_{1}v_{2}v_{3}v_{4}v_{5}]$ is an internal $5$-face. By \autoref{cor-1}, $t_f \leq 2$. If $f$ is not adjacent to any internal $3$-face, then $\mu'(f) \geq \mu(f) - 2 \times \frac{1}{2} = 0$ by \ref{R2}. So we may assume that $f$ is adjacent to at least one internal $3$-face. Since the configurations \autoref{AT345A}--\ref{AT345D} are forbidden, $f$ is adjacent to exactly one internal $3$-face $f^{*}$ and no $4$-faces. If $t_f \leq 1$, then $\mu'(f) \geq \mu(f) - \frac{1}{3} - \frac{1}{2} > 0$ by \ref{R1} and \ref{R2}. Assume $t_f = 2$ and $f^{*} = [uv_{1}v_{2}]$ is an internal $3$-face. If there are some special vertices in $\{u, v_{1}, v_{2}, \dots, v_{5}\}$, then $f$ receives at least $\frac{1}{6}$ from special vertices, and then $\mu'(f) \geq \mu(f) - \frac{1}{3} - (\frac{1}{3} + \frac{1}{2}) + \frac{1}{6} = 0$ by \ref{R1}, \ref{R2}, \ref{R3}, \ref{R4} and \ref{R5}. So we may assume that none of $\{u, v_{1}, v_{2}, \dots, v_{5}\}$ is a special vertex. It follows that $f$ is incident with two $3$-vertices and three $4$-vertices. If neither $v_{1}$ nor $v_{2}$ is a $3$-vertex, then $\mu'(f) \geq \mu(f) - \frac{1}{3} - 2 \times \frac{1}{3} = 0$ by \ref{R1} and \ref{R2}. Without loss of generality, assume that $d(v_{2}) = 3$ and $d(v_{1}) = d(v_{3}) = d(u) = 4$. If $d(v_{4}) = 3$ and $d(v_{5}) = 4$, then it contradicts \autoref{S}. If $d(v_{4}) = 4$ and $d(v_{5}) = 3$, then it contradicts \autoref{TC1}. 

Assume that $f = [v_{1}v_{2}v_{3}v_{4}v_{5}v_{6}]$ is an internal $6$-face. By \autoref{cor-1}, $t_f \leq 3$. 

$\bullet$ $t_f = 3$. Without loss of generality, assume that $v_{1}, v_{3}$ and $v_{5}$ are normal $3$-vertices.

By \autoref{cor-1}, $s_f \le 3$. If $s_f \le 1$, then $\mu'(f) \geq \mu(f) - \frac{1}{3} - 3 \times \frac{1}{2} > 0$ by \ref{R1} and \ref{R2}. 

Assume that $s_f=2$. By symmetry, assume that one of the adjacent internal $3$-face is $[v_{1}v_{2}u]$. By \autoref{TC1}, one vertex in $\{u, v_{2}\}$ is a special vertex. Thus, $\mu'(f) \geq \mu(f) - 2 \times \frac{1}{3} - 3 \times \frac{1}{2} + \frac{1}{6} = 0$ by \ref{R1}, \ref{R2}, \ref{R3}, \ref{R4} and \ref{R5}. 

Assume that $s_f=3$. 

(i) $v_{i}v_{i+1}$ is incident with an internal $3$-face $[v_{i}v_{i+1}u_{i}]$ for $i \in \{1, 3, 5\}$. For each $i\in\{1, 3, 5\}$, by \autoref{TC1}, there is a special vertex in $\{u_{i}, v_{i+1}\}$. Thus $f$ receives at least $\frac{1}{6}$ from $\{u_{i}, v_{i+1}\}$ by \ref{R3}, \ref{R4} and \ref{R5}. Hence, $\mu'(f) \geq \mu(f) - 3 \times \frac{1}{3} - 3 \times \frac{1}{2} + 3 \times \frac{1}{6} = 0$ by \ref{R1}, \ref{R2}, \ref{R3}, \ref{R4} and \ref{R5}. 

(ii) $v_{i}v_{i+1}$ is incident with an internal $3$-face $[v_{i}v_{i+1}u_{i}]$ for $i \in \{1, 2, 5\}$. If $v_{2}$ is a special vertex, then $f$ receives $\frac{1}{3}$ from $v_{2}$. Otherwise, $v_{2}$ is a normal $4$-vertex. By \autoref{TC1}, both $u_{1}$ and $u_{2}$ are special vertices. Then $f$ receives at least $2 \times \frac{1}{6} = \frac{1}{3}$ from $u_{1}$ and $u_{2}$ by \ref{R3}, \ref{R4} and \ref{R5}. In any way, $f$ receives at least $\frac{1}{3}$ from $\{u_1, u_2, v_2\}$. On the other hand, one of $u_{5}$ and $v_{6}$ is also a special vertex, and $f$ receives at least $\frac{1}{6}$ from $\{u_{5}, v_{6}\}$ by \ref{R3}, \ref{R4} and \ref{R5}. Thus, $\mu'(f) \geq \mu(f) - 3 \times \frac{1}{3} - 3 \times \frac{1}{2} + \frac{1}{3} + \frac{1}{6} = 0$. 

$\bullet$ $t_f = 2$. By \autoref{cor-1}, $s_f \le 4$. If $s_f \le 3$, then $\mu'(f) \geq \mu(f) - 3 \times \frac{1}{3} - 2 \times \frac{1}{2} = 0$ by \ref{R1} and \ref{R2}. 

Assume $s_f=4$. We claim that $f$ will receive at least $\frac{1}{3}$ from vertices. If $f$ is incident with a $2$-vertex, then the $2$-vertex must be in $\{x, y\}$, and $f$ receives at least $\frac{1}{3}$ from incident $2$-vertices by \ref{R5}. So we may assume that $f$ is not incident with any $2$-vertex. By symmetry, it suffices to consider five cases. 

(1) The four adjacent internal $3$-faces are $[v_{i}v_{i+1}u_{i}]$ for $1 \leq i \leq 4$. Thus, the two normal $3$-vertices must be $v_{1}$ and $v_{5}$. If one of $v_{2}, v_{3}$ and $v_{4}$ is a special vertex, then $f$ receives $\frac{1}{3}$ from it by \ref{R3}, \ref{R4} and \ref{R5}. So we may assume that $v_{2}, v_{3}$ and $v_{4}$ are normal $4$-vertices. By \autoref{TC1} and \autoref{TC2}, there are at least two special vertices in $\{u_{1}, u_{2}, u_{3}, u_{4}\}$, thus $f$ receives at least $2 \times \frac{1}{6} = \frac{1}{3}$ from these vertices by \ref{R3}, \ref{R4} and \ref{R5}. 

(2) The four adjacent internal $3$-faces are $[v_{i}v_{i+1}u_{i}]$ for $i \in \{1, 2, 3, 5\}$, while $v_{1}$ and $v_{4}$ are normal $3$-vertices. Similarly, if $v_{2}$ or $v_{3}$ is a special vertex, then $f$ receives at least $\frac{1}{3}$ from it. So we may assume that $v_{2}$ and $v_{3}$ are normal $4$-vertices. By \autoref{TC1} and \autoref{TC2}, there are at least two special vertices in $\{u_{1}, u_{2}, u_{3}\}$, thus $f$ receives at least $2 \times \frac{1}{6} = \frac{1}{3}$ from these vertices by \ref{R3}, \ref{R4} and \ref{R5}. 

(3) The four adjacent internal $3$-faces are $[v_{i}v_{i+1}u_{i}]$ for $i \in \{1, 2, 3, 5\}$, while $v_{1}$ and $v_{5}$ are normal $3$-vertices. By \autoref{TC1}, $u_{5}$ or $v_{6}$ is a special vertex; one of $\{v_{2}, v_{3}, v_{4}, u_{1}, u_{2}, u_{3}\}$ is a special vertex. Thus, $f$ receives at least $2 \times \frac{1}{6} = \frac{1}{3}$ from these vertices by \ref{R3}, \ref{R4} and \ref{R5}. 

(4) The four adjacent internal $3$-faces are $[v_{i}v_{i+1}u_{i}]$ for $i \in \{1, 2, 4, 5\}$, while $v_{1}$ and $v_{3}$ are normal $3$-vertices. If $v_{2}$ is a special vertex, then $f$ receives $\frac{1}{3}$ from it by \ref{R3}, \ref{R4} and \ref{R5}. Otherwise, $v_{2}$ is a normal $4$-vertex. By \autoref{TC1}, each of $u_{1}$ and $u_{2}$ is a special vertex, thus $f$ receives at least $2 \times \frac{1}{6} = \frac{1}{3}$ from $\{u_{1}, u_{2}\}$ by \ref{R3}, \ref{R4} and \ref{R5}. 

(5) The four adjacent internal $3$-faces are $[v_{i}v_{i+1}u_{i}]$ for $i \in \{1, 2, 4, 5\}$, while $v_{1}$ and $v_{4}$ are normal $3$-vertices. By \autoref{TC1}, there is at least one special vertex in $\{u_{1}, u_{2}, v_{2}, v_{3}\}$, and there is at least one special vertex in $\{u_{4}, u_{5}, v_{5}, v_{6}\}$. Thus, $f$ receives at least $2 \times \frac{1}{6} = \frac{1}{3}$ from these vertices by \ref{R3}, \ref{R4} and \ref{R5}. 

To sum up, $f$ always receives at least $\frac{1}{3}$ from some vertices in the above five cases. Therefore, $\mu'(f) \geq \mu(f) - 4 \times \frac{1}{3} - 2 \times \frac{1}{2} + \frac{1}{3} = 0$ by \ref{R1} and \ref{R2}. 

$\bullet$ $t_f = 1$. By \autoref{cor-1}, $s_f \le 5$. If $s_f \le 4$, then $\mu'(f) \geq \mu(f) - \frac{1}{2} - 4 \times \frac{1}{3} > 0$. Assume that $s_f=5$ and for $1 \leq i \leq 5$, $[v_{i}v_{i+1}u_{i}]$ is an internal 3-face. Let $X = \{v_{1}, \dots, v_{6}, u_{1}, \dots, u_{5}\}$. By \autoref{TC3}, there is a special vertex in $X$. Therefore, $f$ receives at least $\frac{1}{6}$ from the special vertices in $X$, and $\mu'(f) \geq \mu(f) - \frac{1}{2} - 5 \times \frac{1}{3} + \frac{1}{6} = 0$ by \ref{R3}, \ref{R4} and \ref{R5}. 

$\bullet$ $t_f = 0$. Then $f$ sends nothing to incident vertices, and $\mu'(f) \geq \mu(f) - 6 \times \frac{1}{3} = 0$.

If $f$ is an internal $7^{+}$-face, then $f$ sends out charges by \ref{R1} and \ref{R2}. As $t_f+s_f \le d(f)$, we have 
\[
\mu'(f) \geq \mu(f) - \frac{s_f}{3} - \frac{t_f}{2} \ge \frac{2}{3}d(f) - 4 - \frac{t_f}{6} \geq \frac{7}{12}d(f) - 4 > 0.
\]

This completes the proof of Case 1 of \autoref{thm-main}. 

\subsection{\texorpdfstring{$G$}{G} has no subgraph isomorphic to any configuration in \autoref{COMMONFIGURE} and \autoref{FIGURE-AT48}}

Lemma \ref{LEM-ADJACENTFIVEFACE} below follows easily from the fact that configurations in \autoref{COMMONFIGURE} and \autoref{FIGURE-AT48} are forbidden.

\begin{figure}
\centering
\begin{tikzpicture}
\def\s{1};
\coordinate (V0) at (90:\s);
\coordinate (V1) at (162:\s);
\coordinate (V2) at (234:\s);
\coordinate (V3) at (306:\s);
\coordinate (V4) at (18:\s);
\coordinate (A) at (220:2.5*\s);
\coordinate (B) at (320:2.5*\s);
\draw (V0)--(V1)--(V2)--(V3)--(V4)--cycle;
\draw (V1)--(A)--(B)--(V4);
\node () at (0, 0) {$5$-face};
\node () at (0, 1.5*\s) {$5$-face};
\draw (V2)--($(V2)!0.3!(A)$);
\draw (V2)--($(V2)!0.3!25:(A)$);
\draw (V3)--($(V3)!0.3!(B)$);
\draw (V3)--($(V3)!0.3!-25:(B)$);
\node[circle, inner sep = 1.5, fill, draw] () at (V0) {};
\node[circle, inner sep = 1.5, fill = white, draw] () at (V1) {};
\node[circle, inner sep = 1.5, fill = white, draw] () at (V2) {};
\node[circle, inner sep = 1.5, fill = white, draw] () at (V3) {};
\node[circle, inner sep = 1.5, fill = white, draw] () at (V4) {};
\node[circle, inner sep = 1.5, fill = white, draw] () at (A) {};
\node[circle, inner sep = 1.5, fill = white, draw] () at (B) {};
\end{tikzpicture}
\caption{Two adjacent $5$-faces. The solid vertex is a $2$-vertex in $G$.}
\label{ADJACENTFIVEFACE}
\end{figure}
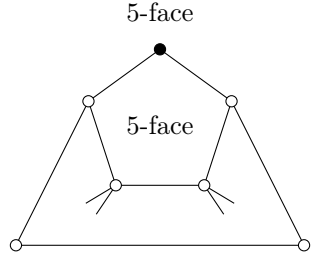

\begin{lemma}\label{LEM-ADJACENTFIVEFACE}
If two $5$-faces have two consecutive common edges on their boundaries, then one of the $5$-face is the outer face $f_{0}$ (see \autoref{ADJACENTFIVEFACE}). 
\end{lemma}

Now we calculate the final charge of internal $5^+$-faces.

Assume $f$ is an internal $d$-face. If $f$ is incident with a $2$-vertex, then the $2$-vertex belongs to $\{x, y\}$, and $f$ is adjacent to at most $d - 2$ internal faces. By \ref{R5}, $f$ receives $\frac{1}{3}$ from each of $x$ and $y$. By \ref{R6}, $f$ receives $\frac{1}{3}$ via each common edge with the outer face $f_{0}$. By \ref{R1} and \ref{RR7}, $f$ sends at most $\frac{1}{3}$ to each adjacent internal face. By \ref{R2}, $f$ sends at most $\frac{1}{2}$ to each incident normal $3$-vertex. Thus, $\mu'(f) \geq d - 4 + 2 \times \frac{1}{3} + 2 \times \frac{1}{3} - (d - 2) \times \frac{1}{3} - \lfloor \frac{d}{2} \rfloor \times \frac{1}{2} \geq \frac{5d - 24}{12} > 0$. 

Assume that $f$ is not incident with any $2$-vertex. By \autoref{LEM-ADJACENTFIVEFACE}, there are no adjacent internal $5$-faces. By \autoref{a3}, $f$ is adjacent to at most $d - t_f$ internal $5$-faces.

$\blacksquare$\quad$ \bm{d = 5}$. Assume that $f = [v_{1}v_{2}v_{3}v_{4}v_{5}]$. Since adjacent triangles and a triangle normally adjacent to a $7$-cycle are forbidden, $s_f \le 2$. By \autoref{cor-1}, $t_f \leq 2$. It follows that $f$ is incident with at most two minor $3$-vertices.

If $s_f=0$, then $\mu'(f) \geq \mu(f) - 2 \times \frac{1}{2} = 0$ by \ref{R2}. 

Assume $s_f \ge 1$. Since \autoref{COMMONFIGURE} and \autoref{AT48C} are forbidden, $f$ is not adjacent to any $4$-face. It follows that every face adjacent to $f$ is a $3$-face or a $6^{+}$-face. Thus, $f$ is adjacent to at least three $6^{+}$-faces (the number of adjacent $6^{+}$-faces is counted by the number of common edges). If $f$ is incident with at most one minor $3$-vertex, then $\mu'(f) \geq 5 - 4 - 2 \times \frac{1}{3} - (\frac{1}{2} + \frac{1}{3}) + 3 \times \frac{1}{6} = 0$ by \ref{R1}, \ref{R2} and \ref{RR7}. Assume $f$ is incident with exactly two minor $3$-vertices. That is $t_f=2$ and $s_f=2$. By symmetry, we have three subcases to consider: 

$\bullet$ $f$ is adjacent to two internal $3$-faces $[v_{1}v_{2}u_{1}]$, $[v_{3}v_{4}u_{3}]$, and $v_{1}, v_{3}$ are minor $3$-vertices.

$\bullet$ $f$ is adjacent to two internal $3$-faces $[v_{1}v_{2}u_{1}]$, $[v_{3}v_{4}u_{3}]$, and $v_{1}, v_{4}$ are minor $3$-vertices.

$\bullet$ $f$ is adjacent to two internal $3$-faces $[v_{1}v_{2}u_{1}]$, $[v_{2}v_{3}u_{2}]$, and $v_{1}, v_{3}$ are minor $3$-vertices. 

\noindent By \autoref{TC1} and \autoref{TC2}, the two $3$-faces are incident with at least one special vertex. By \ref{R3}, \ref{R4} and \ref{R5}, $f$ receives at least $\frac{1}{6}$ from these special vertices. Hence, $\mu'(f) \geq 5 - 4 + \frac{1}{6} + 3 \times \frac{1}{6} - 2 \times \frac{1}{3} - 2 \times \frac{1}{2} = 0$.

$\blacksquare$\quad $\bm{d = 6}$. Assume that $f = [v_{1}v_{2}v_{3}v_{4}v_{5}v_{6}]$. If $s_f=0$, then it sends at most $\frac{1}{2}$ to each incident normal $3$-vertex, and sends $\frac{1}{6}$ to each adjacent $5$-face, thus $\mu'(f) \geq 6 - 4 - t_f \times \frac{1}{2} - (6 - t_f) \times \frac{1}{6} = 1 - \frac{t_f}{3} \geq 0$ by \ref{R2} and \ref{RR7}. 

Suppose that $f$ is adjacent to an internal $3$-face. Then they are normally adjacent. Since the configurations in \autoref{COMMONFIGURE} and \autoref{AT48C} are forbidden, $s_f=1$. By \autoref{cor-1}, $t_f \leq 3$. If $t_f\le 2$, then $\mu'(f) \geq 6 - 4 - \frac{1}{3} - t_f \times \frac{1}{2} - (6 - t_f) \times \frac{1}{6} = \frac{2 - t_f}{3} \geq 0$ by \ref{R1}, \ref{R2} and \ref{RR7}. 

Assume $t_f=3$ and the $3$-face is $[uv_{1}v_{2}]$. By \autoref{a3}, we may assume $v_{1}, v_{3}$ and $v_{5}$ are the three normal $3$-vertices. By \autoref{TC1}, there is a special vertex in $\{u, v_{2}\}$, thus $f$ receives at least $\frac{1}{6}$ from $\{u, v_{2}\}$. Since the configurations in \autoref{COMMONFIGURE} and \autoref{FIGURE-AT48} are all forbidden, $v_{5}$ cannot be incident with an internal $4^{-}$-face. Thus, $f$ is incident with at most two minor $3$-vertices, which implies that $\mu'(f) \geq 6 - 4 - (2 \times \frac{1}{2} + \frac{1}{3}) - \frac{1}{3} - (6 - 3) \times \frac{1}{6} + \frac{1}{6} = 0$.

$\blacksquare$\quad $\bm{d = 7}$. Let $f$ be a $7$-face. As \autoref{AT48C} is forbidden, $s_f = 0$. By \autoref{cor-1}, $t_f \le 3$. By \ref{R2}, $f$ sends at most $\frac{1}{2}$ to each incident normal  $3$-vertex. By \ref{RR7}, $f$ sends $\frac{1}{6}$ to each adjacent internal $5$-face. Hence, $\mu'(f) \geq 7 - 4 - t_f \times \frac{1}{2} - (7 - t_f) \times \frac{1}{6} = \frac{11 - 2t_f}{6} > 0$. 

$\blacksquare$\quad $\bm{d \ge 8}$. Let $f$ be a $8^+$-face. Then  $f$ sends at most $\frac{1}{2}$ to each incident normal $3$-vertex, and $\frac{1}{3}$ to each adjacent internal $3$-face, and $\frac{1}{6}$ to each adjacent internal $5$-face. Combining with \autoref{cor-1}, we have that

\[
\mu'(f) \geq d - 4 - t_f\times\frac{1}{2} - s_f\times\frac{1}{3}  -(d-s_f)\times\frac{1}{6} =\frac{5}{6}d -\frac{1}{2}t_f-\frac{1}{6}s_f-4\geq \frac{d}{2}-4 \geq 0.
\]


This completes the proof of Case 2. 

\subsection{\texorpdfstring{$G \in \mathcal{G}_{4,9}$}{G4,9}}

\begin{lemma}\label{TRIANGULAREDGE}
A $5$-cycle contains at most three triangular edges. 
\end{lemma}
\begin{proof}
Assume $[x_{1}x_{2}x_{3}x_{4}x_{5}]$ is a $5$-cycle, and $[x_{1}x_{2}x_{6}], [x_{2}x_{3}x_{7}], [x_{3}x_{4}x_{8}]$ and $[x_{4}x_{5}x_{9}]$ are four triangles. Since there is no $4$-cycle in $G$, $x_{1}, x_{2}, \dots, x_{9}$ are nine distinct vertices. Thus, $[x_{1}x_{6}x_{2}x_{7}x_{3}x_{8}x_{4}x_{9}x_{5}]$ is a $9$-cycle, a contradiction. 
\end{proof}

\begin{lemma}\label{FIVEFACE}
Let $f = [x_{1}x_{2}x_{3}x_{4}x_{5}]$ and $g = [x_{5}x_{1}uvw]$ be two adjacent $5$-faces. If $d(x_{1}) \geq 3$ and $d(x_{5}) \geq 3$, then $f$ and $g$ are normally adjacent, and neither $x_{2}x_{3}$ nor $x_{3}x_{4}$ is adjacent to a $3$-face. Moreover, if $x_{1}x_{2}$ is incident with a $3$-face, then $x_{1}$ is a $3$-vertex and the $3$-face is $[x_{1}x_{2}u]$. 
\end{lemma}
\begin{proof}
Since $d(x_{1}) \geq 3$ and $d(x_{5}) \geq 3$, we have that $ x_{2}\ne u$ and $x_{4}\ne w$. Since $G$ has no $4$-cycle, $x_{1}, x_{2}, \dots, x_{5}, u, v, w$ are distinct. Therefore, $f$ and $g$ are normally adjacent. 

By the symmetry of $x_{2}x_{3}$ and $x_{3}x_{4}$, suppose that $x_{2}x_{3}$ is incident with a $3$-face $[x_{2}x_{3}x_{7}]$. Since there are no $4$-cycles in $G$, $x_{7}$ is not incident with $f$ or $g$. Thus, $[x_{5}x_{4}x_{3}x_{7}x_{2}x_{1}uvw]$ is a $9$-cycle, a contradiction. Hence, neither $x_{2}x_{3}$ nor $x_{3}x_{4}$ is incident with a $3$-face. 

Let $x_{1}x_{2}$ be incident with a $3$-face $[x_{1}x_{2}x_{6}]$. Since $f$ has no chord, $x_{6} \notin \{x_{3}, x_{4}, x_{5}, v, w\}$. If $x_{6} \neq u$, then $[x_{5}x_{4}x_{3}x_{2}x_{6}x_{1}uvw]$ is a $9$-cycle, a contradiction. Thus $x_{6} = u$ and $x_{1}$ is a $3$-vertex. 
\end{proof}

\begin{lemma}\label{SIXFACE}
Let $f = [x_{1}x_{2}x_{3}x_{4}x_{5}]$ and $g = [x_{5}x_{1}upqw]$ be two adjacent faces. If $d(x_{1}) \geq 3$ and $d(x_{5}) \geq 3$, then $\{u, w\} \cap \{x_{1}, \dots, x_{5}\} = \emptyset$, while $\{p, q\} \cap \{x_{2}, x_{3}, x_{4}\} = \{p\} = \{x_{2}\}$ or $\{p, q\} \cap \{x_{2}, x_{3}, x_{4}\} = \{q\} = \{x_{4}\}$. 
\end{lemma}
\begin{proof}
Since $G$ has no $9$-cycle, $\{x_2, x_3, x_4\}\cap\{u, p, q, w\}\ne\emptyset$. For $d(x_{1}) \geq 3$ and $d(x_{5}) \geq 3$, we have that $x_{2}\ne u$ and $x_{4}\ne w$. Note that there are no $4$-cycles, it follows that $\{x_2, x_3, x_4\}\cap\{u, w\}=\emptyset$, $x_{3} \notin \{p, q\}$, $x_4\ne p$ and $x_2\ne q$. Therefore, $\{p, q\} \cap \{x_{2}, x_{4}\} = \{p\} = \{x_{2}\}$ or $\{p, q\} \cap \{x_{2}, x_{4}\} = \{q\} = \{x_{4}\}$. 
\end{proof}

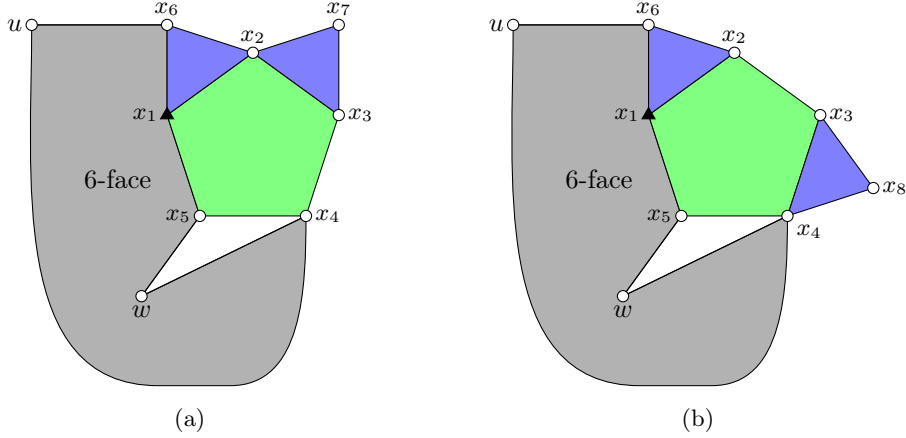
\begin{figure}
\centering
\subcaptionbox{\label{N49FIVE1}}[0.4\linewidth]{\begin{tikzpicture}
\def\s{1.2}
\foreach \ang in {1, 2, 3, 4, 5}
{
\def\pointname{v\ang}
\coordinate (\pointname) at ($(\ang*360/5-54:\s)$);}
\coordinate (NE) at ($(v1)+(v2)$);
\coordinate (NW) at ($(v2)+(v3)$);
\coordinate (u) at ($(NW)+(-1.5*\s,0)$);
\coordinate (w) at ($2.1*(v4)$);
\draw[fill = gray!60] (u) to [out=-90, in = 180] ($(u)+(1.4*\s, -4*\s)$) to ($(u)+(2.2*\s, -4*\s)$) to [out=0, in= -90] (v5)--(w)--(v4)--(v3)--(NW)--(u)--cycle;
\draw[fill = green!50] (v1) node[right]{\small$x_{3}$}--(v2)node[above]{\small$x_{2}$}--(v3)node[left]{\small$x_{1}$}--(v4)node[left]{\small$x_{5}$}--(v5)node[right]{\small$x_{4}$}--cycle;
\draw[fill = blue!50] (v1)--(NE)node[above]{\small$x_{7}$}--(v2)--cycle;
\draw[fill = blue!50] (v2)--(NW)node[above]{\small$x_{6}$}--(v3)--cycle;
\draw (NW)--(u)node[left]{$u$};
\draw (v4)--(w)node[below]{$w$};
\draw (w)--(v5);
\node[circle, inner sep = 1.5, fill = white, draw] () at (v1) {};
\node[circle, inner sep = 1.5, fill = white, draw] () at (v2) {};
\node[regular polygon, regular polygon sides=3, inner sep = 1, fill, draw] () at (v3) {};
\node[circle, inner sep = 1.5, fill = white, draw] () at (v4) {};
\node[circle, inner sep = 1.5, fill = white, draw] () at (v5) {};
\node[circle, inner sep = 1.5, fill = white, draw] () at (NE) {};
\node[circle, inner sep = 1.5, fill = white, draw] () at (NW) {};
\node[circle, inner sep = 1.5, fill = white, draw] () at (u) {};
\node[circle, inner sep = 1.5, fill = white, draw] () at (w) {};
\node () at (-1.5*\s, -0.4*\s) {$6$-face};
\end{tikzpicture}}
\subcaptionbox{\label{N49FIVE2}}[0.4\linewidth]{\begin{tikzpicture}
\def\s{1.2}
\foreach \ang in {1, 2, 3, 4, 5}
{
\def\pointname{v\ang}
\coordinate (\pointname) at ($(\ang*360/5-54:\s)$);}
\coordinate (SE) at ($(v1)+(v5)$);
\coordinate (NW) at ($(v2)+(v3)$);
\coordinate (u) at ($(NW)+(-1.5*\s,0)$);
\coordinate (w) at ($2.1*(v4)$);
\draw[fill = gray!60] (u) to [out=-90, in = 180] ($(u)+(1.4*\s, -4*\s)$) to ($(u)+(2.2*\s, -4*\s)$) to [out=0, in= -90] (v5)--(w)--(v4)--(v3)--(NW)--(u)--cycle;
\draw[fill = green!50] (v1) node[right]{\small$x_{3}$}--(v2)node[above]{\small$x_{2}$}--(v3)node[left]{\small$x_{1}$}--(v4)node[left]{\small$x_{5}$}--(v5)node[below right]{\small$x_{4}$}--cycle;
\draw[fill = blue!50] (v1)--(SE)node[right]{\small$x_{8}$}--(v5)--cycle;
\draw[fill = blue!50] (v2)--(NW)node[above]{\small$x_{6}$}--(v3)--cycle;
\draw (NW)--(u)node[left]{$u$};
\draw (v4)--(w)node[below]{$w$};
\draw (w)--(v5);
\node[circle, inner sep = 1.5, fill = white, draw] () at (v1) {};
\node[circle, inner sep = 1.5, fill = white, draw] () at (v2) {};
\node[regular polygon, regular polygon sides=3, inner sep = 1, fill, draw] () at (v3) {};
\node[circle, inner sep = 1.5, fill = white, draw] () at (v4) {};
\node[circle, inner sep = 1.5, fill = white, draw] () at (v5) {};
\node[circle, inner sep = 1.5, fill = white, draw] () at (SE) {};
\node[circle, inner sep = 1.5, fill = white, draw] () at (NW) {};
\node[circle, inner sep = 1.5, fill = white, draw] () at (u) {};
\node[circle, inner sep = 1.5, fill = white, draw] () at (w) {};
\node () at (-1.5*\s, -0.4*\s) {$6$-face};
\end{tikzpicture}}
\caption{Some local structures around $5$-face.}
\label{}
\end{figure}

\begin{lemma}\label{SPECIAL6FACE1}
Let $f = [x_{1}x_{2}x_{3}x_{4}x_{5}]$ be a $5$-face adjacent to two $3$-faces, that are either $[x_{1}x_{2}x_{6}]$ and $[x_{2}x_{3}x_{7}]$, or $[x_{1}x_{2}x_{6}]$ and $[x_{3}x_{4}x_{8}]$ (see \autoref{N49FIVE1} and \autoref{N49FIVE2}). If $d(x_{1}) = 3$, $d(x_{5}) \geq 3$ and $d(x_{6}) \geq 3$, and $x_{5}x_{1}x_{6}$ is incident with a $6^{-}$-face $g$, then $g$ is a $6$-face $[x_{5}x_{1}x_{6}uvw]$, where $\{u, w\} \cap \{x_{1}, x_{2}, \dots, x_{8}\} = \emptyset$, $v = x_{4}$ and $d(x_{4}) \geq 4$ ($d(x_{4}) \geq 5$ for the case of \autoref{N49FIVE2}). 
\end{lemma}
\begin{proof}
We only consider the case of \autoref{N49FIVE1} here, the case of \autoref{N49FIVE2} is quite similar. Suppose that $g = [x_{5}x_{1}x_{6}u\dots w]$. Since $d(x_{5}) \geq 3$ and $d(x_{6}) \geq 3$, $x_{1}, x_{2}, x_{6}, u$ are four distinct vertices, and $x_{1}, x_{4}, x_{5}, w$ are four distinct vertices. As there is no $4$-cycle in $G$, $x_{1}, x_{2}, \dots, x_{7}, u, w$ are distinct. It follows that $g$ must be a $5$- or $6$-face. If $g$ is a $5$-face, then $g = [x_{5}x_{1}x_{6}uw]$ and $[x_{5}x_{4}x_{3}x_{7}x_{2}x_{1}x_{6}uw]$ is a $9$-cycle, a contradiction. Let $g = [x_{5}x_{1}x_{6}uvw]$ be a $6$-face. If $v \notin \{x_{2}, x_{3}, x_{4}\}$, then $[uvwx_{5}x_{4}x_{3}x_{2}x_{1}x_{6}]$ is a $9$-cycle, a contradiction. If $v = x_{2}$, then $[ux_{6}x_{1}x_{2}]$ is a $4$-cycle, a contradiction. If $v = x_{3}$, then $[ux_{6}x_{2}x_{3}]$ is a $4$-cycle, a contradiction. Hence, $v = x_{4}$ and $[x_{4}x_{5}w]$ is a triangle. 
\end{proof}

\begin{lemma}\label{GOOD}
Let $f = [x_{1}x_{2}x_{3}x_{4}x_{5}]$ be a $5$-face adjacent to two $3$-faces $[x_{1}x_{2}x_{6}]$ and $[x_{3}x_{4}x_{8}]$. If $d(x_{2}) = 3$, $d(x_{3}) \geq 4$ and $d(x_{6}) \geq 3$, then $x_{3}x_{2}x_{6}$ is incident with a $7^{+}$-face.
\end{lemma}
\begin{proof}
Suppose that $x_{3}x_{2}x_{6}$ is incident with a face $g = [x_{3}x_{2}x_{6}u\dots w]$. Since $d(x_{3}) \geq 4$ and $d(x_{6}) \geq 3$, we have that $x_{2}, x_{3}, x_{4}, x_{8}, w$ are five distinct vertices, and $x_{1}, x_{2}, x_{6}, u$ are four distinct vertices. Since there are no $4$-cycles, we have that $x_{1}, x_{2}, \dots, x_{6}, x_{8}, u, w$ are distinct. It follows that $g$ must be a $5^{+}$-face. If $g$ is a $5$-face, then $g = [x_{3}x_{2}x_{6}uw]$ and $[x_{3}x_{8}x_{4}x_{5}x_{1}x_{2}x_{6}uw]$ is a $9$-cycle, a contradiction. Let $g$ be a $6$-face $[x_{3}x_{2}x_{6}uvw]$. If $v \notin \{x_{1}, x_{4}, x_{5}\}$, then $[uvwx_{3}x_{4}x_{5}x_{1}x_{2}x_{6}]$ is a $9$-cycle, a contradiction. If $v = x_{1}$, then $[ux_{6}x_{2}x_{1}]$ is a $4$-cycle, a contradiction. If $v = x_{4}$, then $[wx_{3}x_{8}x_{4}]$ is a $4$-cycle, a contradiction. If $v = x_{5}$, then $[ux_{6}x_{1}x_{5}]$ is a $4$-cycle, a contradiction. Therefore, $x_{3}x_{2}x_{6}$ is incident with a $7^{+}$-face. 
\end{proof}

\begin{lemma}\label{GOODFACE}
Let $f = [x_{1}x_{2}x_{3}\dots]$ be a $7^{+}$-face. If $x_{2}$ is a normal $3$-vertex, then at most one of $x_{1}x_{2}$ and $x_{2}x_{3}$ is incident with a good $5$-face. 
\end{lemma}
\begin{proof}
Suppose to the contrary that $x_{1}x_{2}$ is incident with a good $5$-face $g_{1} = [x_{1}x_{2}v_{3}v_{4}v_{5}]$ and $x_{2}x_{3}$ is incident with a good $5$-face $g_{2} = [x_{3}x_{2}v_{3}u_{4}u_{5}]$. Note that $g_{1}$ and $g_{2}$ are all internal faces. By \autoref{delta}, $v_{3}$ cannot be a $2$-vertex. By \autoref{FIVEFACE}, $g_{1}$ and $g_{2}$ are normally adjacent. Moreover, $v_{3}$ is a $3$-vertex, and $g_{3} = [v_{3}v_{4}u_{4}]$ is an internal $3$-face. It is observed that $g_{1}, g_{2}$ and $g_{3}$ are all internal faces. It follows that $v_{3}$ does not belong to $\{x, y\}$, but this contradicts \autoref{a3}. 
\end{proof}

Let $\tau(\rightarrow f)$ be the number of charges that $f$ receives from other elements. 
\begin{claim}\label{1F2V1M}
If $f$ is an internal $5$-face and $s_f=1$, then $\tau(\rightarrow f) \geq \frac{1}{3}$.
\end{claim}
\begin{proof}
Let $f = [v_{1}v_{2}v_{3}v_{4}v_{5}]$ be an internal $5$-face, and let $[v_{1}v_{2}v_{6}]$ be an internal $3$-face. Since $f$ has no chord, $v_{1}, v_{2}, \dots, v_{6}$ are six distinct vertices. If $v_{i} \in \{x, y\}$ for any $1 \leq i \leq 6$, then $v_{i}$ sends $\frac{1}{3}$ to $f$ by \ref{R5}, we are done. Assume $\{v_{1}, v_{2}, \dots, v_{6}\} \cap \{x, y\} = \emptyset$. By \autoref{delta}, $d(v_{i}) \geq 3$ for $1 \leq i \leq 6$. 

Next, we show that $f$ is adjacent to a special face. By the hypothesis, neither $v_{3}v_{4}$ nor $v_{4}v_{5}$ is incident with an internal $4^{-}$-face. By \autoref{FIVEFACE}, neither $v_{3}v_{4}$ nor $v_{4}v_{5}$ is incident with a $5$-face. If $v_{3}v_{4}$ or $v_{4}v_{5}$ is incident with an internal $7^{+}$-face or $f_0$, we are done. So we may assume that each of $v_{3}v_{4}$ and $v_{4}v_{5}$ is incident with an internal $6$-face. By \autoref{SIXFACE}, $v_{3}v_{4}$ is incident with a $6$-face $[v_{3}v_{4}upv_{2}w]$. If $[v_{2}v_{3}w]$ bounds a $3$-face, then $d(w) = 2$ and $v_{2}v_{3}$ is incident with the outer face $[v_{2}v_{3}w]$, we are done. Hence, we can assume that $v_{2}v_{3}$ is not incident with a $3$-face. By \autoref{FIVEFACE}, $v_{2}v_{3}$ cannot be incident with a $5$-face. Since there are no $9$-cycles, $v_{2}v_{3}$ cannot be incident with a $6$-face. Hence, $v_{2}v_{3}$ is incident with a $7^{+}$-face. Therefore, $f$ is adjacent to at least one special face in any case. By \ref{R6} and \ref{RRR7}, $f$ receives $\frac{1}{3}$ from each adjacent special face, thus $\tau(\rightarrow f) \geq \frac{1}{3}$. 
\end{proof}

\begin{claim}\label{2F1M}
Let $f$ be an internal $5$-face and $s_f=2$. If $f$ is incident with one minor $3$-vertex, then $\tau(\rightarrow f) \geq \frac{1}{3}$.
\end{claim}
\begin{proof}
Assume that $f = [x_{1}x_{2}x_{3}x_{4}x_{5}]$. If $x$ or $y$ is incident with $f$ or one of the adjacent $3$-faces, then it sends at least $\frac{1}{3}$ to $f$ by \ref{R5}. So we may assume that neither $x$ nor $y$ is incident with $f$ or the adjacent $3$-faces. Now we show that $f$ is adjacent to at least one $7^+$-face sending $\frac{1}{3}$ to $f$ by \ref{R6} and \ref{RRR7}.

\textbf{Case 1.} Let $[x_{1}x_{2}x_{6}]$ and $[x_{2}x_{3}x_{7}]$ be internal $3$-faces, and let $x_{1}$ be a minor $3$-vertex. By \autoref{delta} and \autoref{a3}, $d(x_{5}) \geq 4$ and $d(x_{6}) \geq 4$. By \autoref{SPECIAL6FACE1}, if $x_{5}x_{1}x_{6}$ is incident with a $6^{-}$-face, then $[x_{4}x_{5}w]$ is a triangle but it does not bound a $3$-face, thus $x_{4}x_{5}$ is incident with a $7^{+}$-face. Hence, either $x_{5}x_{1}x_{6}$ or $x_{4}x_{5}$ is incident with a $7^{+}$-face. 

\textbf{Case 2.} Let $[x_{1}x_{2}x_{6}]$ and $[x_{3}x_{4}x_{8}]$ be internal $3$-faces, and let $x_{1}$ be a minor $3$-vertex. By \autoref{delta}, \autoref{a3} and \autoref{SPECIAL6FACE1}, we also get that either $x_{5}x_{1}x_{6}$ or $x_{4}x_{5}$ is incident with a $7^{+}$-face. 

\textbf{Case 3.} Let $[x_{1}x_{2}x_{6}]$ and $[x_{3}x_{4}x_{8}]$ be internal $3$-faces, and let $x_{2}$ be a minor $3$-vertex. By \autoref{delta} and \autoref{a3}, $d(x_{3}) \geq 4$ and $d(x_{6}) \geq 4$. By \autoref{GOOD}, $x_{2}x_{3}$ is incident with a $7^{+}$-face. 
\end{proof}

\begin{claim}\label{2F2M}
Let $f$ be an internal $5$-face and $s_f\ge2$. If $f$ is incident with two minor $3$-vertices, then $\tau(\rightarrow f) \geq 1$.
\end{claim}
\begin{proof}
Assume $f = [x_{1}x_{2}x_{3}x_{4}x_{5}]$. If $x_{i}$ is a $2$-vertex, then $x_{i} \in \{x, y\}$ and $x_{i-1}x_{i}x_{i+1}$ is incident with the outer face $f_{0}$. By \ref{R5}, $f$ receives $\frac{1}{3}$ from each of $x$ and $y$. By \ref{R6}, $f$ receives $\frac{1}{3}$ via each of $x_{i-1}x_{i}$ and $x_{i}x_{i+1}$. Thus, $\tau(\rightarrow f) \geq 2\times \frac{1}{3} + 2 \times \frac{1}{3} > 1$. So we may assume that $d(x_{i}) \geq 3$ for any $1 \leq i \leq 5$. Denote the adjacent face incident with $x_ix_{i+1}$ by $g_i$ for $i\in\{1, 2, 3, 4, 5\}$.

\textbf{Case 1.} Let $[x_{1}x_{2}x_{6}]$ and $[x_{2}x_{3}x_{7}]$ be internal $3$-faces,  and let $x_{1}$ and $x_{3}$ be minor $3$-vertices. Suppose that $x_{6}$ is a $2$-vertex. It follows that $\{x_{2}, x_{6}\} = \{x, y\}$ and $g_5=f_{0}$. By \ref{R5}, $f$ receives $\frac{1}{3}$ from each of $x_{2}$ and $x_{6}$. By \ref{R6}, $f$ receives at least $\frac{1}{3}$ from the outer face $f_{0}$. Thus, $\tau(\rightarrow f) \geq 3 \times \frac{1}{3} = 1$. 

So we may assume that $d(x_{6}) \geq 3$, and by symmetry, $d(x_{7}) \geq 3$. Firstly, we claim that $f$ receives at least $\frac{1}{3}$ from $\{x_{2}, x_{6}, x_{7}\}$. If $x_{2}$ is a special vertex, then $f$ receives $\frac{1}{3}$ from $x_{2}$ by \ref{R3}, \ref{R4} and \ref{R5}. So we may assume that $x_{2}$ is a normal $4$-vertex. It follows from \autoref{TC1} that both $x_{6}$ and $x_{7}$ are special vertices. By \ref{R3}, \ref{R4} and \ref{R5}, $f$ receives at least $\frac{1}{6}\times2=\frac{1}{3}$ from $x_{6}$ and $x_{7}$.

Next, we show that $f$ is adjacent to at least two special faces. Since $f$ receives at least $2 \times \frac{1}{3}=\frac{2}{3}$ from adjacent special faces by \ref{R6} and \ref{RRR7}, we are done. By \autoref{SPECIAL6FACE1}, we get that both $g_3$ and $g_5$ are $6^+$-faces, and $g_3$, $g_5$ cannot be $6$-face simultaneously. If both $g_3$ and $g_5$ are $7^+$-faces, then we are done. By symmetry, assume that $g_{5}$ is a $6$-face and $g_{3}$ is a $7^+$-face. It follows that $g_4$ is the outer $3$-face or a $7^+$-face. That is, $g_3$ and $g_4$ are the special faces, we are done.

\textbf{Case 2.} Let $[x_{1}x_{2}x_{6}]$ and $[x_{3}x_{4}x_{8}]$ be internal $3$-faces,  and let $x_{1}$ and $x_{4}$ be minor $3$-vertices. Similar to Case~1, we may assume that $d(x_{6}) \geq 3$ and $d(x_{8}) \geq 3$. Note that $x_{1}$ and $x_{4}$ are $3$-vertices. Since there are no $4$-cycles, neither $g_{4}$ nor $g_{5}$ is a $4^{-}$-face. By \autoref{FIVEFACE}, neither $g_{4}$ nor $g_{5}$ is a $5$-face. By \autoref{SIXFACE} and \autoref{SPECIAL6FACE1}, neither $g_{4}$ nor $g_{5}$ is a $6$-face. So both $g_{4}$ and $g_{5}$ are $7^{+}$-faces. Thus, $f$ receives at least $\frac{1}{3}\times2=\frac{2}{3}$ from these $7^{+}$-faces. Next we show that $f$ will receive at least $\frac{1}{3}$ from others.

If $g_{2}$ is a $7^{+}$-face, then we are done. By \autoref{SIXFACE}, $g_{2}$ cannot be a $6$-face. Assume $g_{2}$ is a $5$-face. By \autoref{FIVEFACE}, $d(x_{2}) = d(x_{3}) = 3$. By \autoref{a3}, we have that $\{x_{2}, x_{3}\} = \{x, y\}$. By \ref{R5}, $f$ receives $\frac{1}{3}$ from each of $x_{2}$ and $x_{3}$, we are done. It is clear that $g_{2}$ cannot be a $4$-face. Suppose that $g_{2}$ is a $3$-face $[x_{2}x_{3}x_{7}]$. If there is one special vertex in $\{x_{2}, x_{3}\}$, then we are done by \ref{R3}, \ref{R4} and \ref{R5}. So we may assume that both $x_{2}$ and $x_{3}$ are normal $4$-vertices. By \autoref{TC1} and \autoref{TC2}, at least two of $x_{6}, x_{7}$ and $x_{8}$ are special vertices, thus $f$ receives at least $2 \times \frac{1}{6} = \frac{1}{3}$ from these special vertices, we are done.

\textbf{Case 3.} Let $[x_{1}x_{2}x_{6}]$ and $[x_{3}x_{4}x_{8}]$ be internal $3$-faces, and let $x_{1}$ and $x_{3}$ be minor $3$-vertices. Similar to Case~1, assume $d(x_{6}) \geq 3$ and $d(x_{8}) \geq 3$. By \autoref{TC1}, one of $\{x_{2}, x_{6}\}$ is a special vertex. By \ref{R3}, \ref{R4} and \ref{R5}, $f$ receives at least $\frac{1}{6}$ from $\{x_{2}, x_{6}\}$.

Since there are no $4$-cycles, we have that $g_{2}$ cannot be a $4^{-}$-face. Suppose that $g_{2}$ is a $5$-face. By \autoref{FIVEFACE}, we have that $d(x_{2}) = d(x_{3}) = 3$. By \autoref{a3}, $x_{2}$ belongs to $\{x, y\}$. As a consequence, $\{x_{2}, x_{6}\} = \{x, y\}$ and $g_{2}$ is the outer face $f_{0}$. By \ref{R5} and \ref{R6}, $\tau(\rightarrow f) \geq 2 \times \frac{1}{3} + \frac{1}{3} = 1$, we are done. By \autoref{SIXFACE}, $g_{2}$ cannot be a $6$-face. Thus, we may assume that $g_{2}$ is a $7^{+}$-face. By \ref{RRR7}, $f$ receives $\frac{1}{3}$ from $g_2$. 

Next we show that $f$ receives at least $\frac{1}{2}$ from others. By \autoref{FIVEFACE} and \autoref{SIXFACE}, $g_{4}$ cannot be a $5$- or $6$-face. Thus, $g_{4}$ is a $3$- or $7^{+}$-face. Suppose that $g_{4}$ is a $3$-face $[x_{4}x_{5}x_{9}]$. If $x_{9}$ is a $2$-vertex, then $\{x, y\} \subset \{x_{4}, x_{5}, x_{9}\}$, and then $f$ receives at least $2 \times \frac{1}{3}>\frac{ 1}{2}$ from $x$ and $y$ by \ref{R5}. So we may assume that $d(x_{9}) \geq 3$. By \autoref{FIVEFACE} and \autoref{SIXFACE}, $g_{5}$ is a $7^{+}$-face sending $\frac{1}{3}$ to $f$. By \autoref{TC1}, there is a special vertex in $\{x_{4}, x_{5}, x_{8}, x_{9}\}$ sending at least $\frac{1}{6}$ to $f$. Thus, $f$ receives at least $\frac{1}{6} +\frac{1}{3} =\frac{ 1}{2}$ from $g_5$ and the special vertex. Suppose that $g_{4}$ is a $7^{+}$-face. If $g_{5}$ is also a $7^{+}$-face, then $f$ receives at least $2 \times \frac{1}{3}>\frac{ 1}{2}$ from $g_4$ and $g_5$, we are done. So we may assume that $g_{5}$ is a $6^{-}$-face. By \autoref{SPECIAL6FACE1}, $d(x_{4}) \geq 5$. By \ref{R3}, \ref{R4} and \ref{R5}, $f$ receives at least $\frac{1}{6}$ from $x_{4}$. Therefore, $f$ still receives at least $\frac{1}{6} +\frac{1}{3} =\frac{ 1}{2}$ from $g_4$ and $x_4$. 
\end{proof}

\begin{claim}\label{2F2V1M}
Let $f$ be an internal $5$-face and $s_f=2$. If $t_f=2$, and exactly one of the two normal $3$-vertices is minor, then $\tau(\rightarrow f) \geq \frac{1}{2}$.
\end{claim}
\begin{proof}
Assume $f= [x_{1}x_{2}x_{3}x_{4}x_{5}]$. By the definition of normal $3$-vertex and minor $3$-vertex, we only need to consider two cases.

\textbf{Case 1.} Let $[x_{1}x_{2}x_{6}]$ and $[x_{2}x_{3}x_{7}]$ be internal $3$-faces, and let $x_{1}$ and $x_{4}$ be normal $3$-vertices. If $x_{5}$ or $x_{6}$ is a $2$-vertex, then $x_{5}$ or $x_{6}$ belongs to $\{x, y\}$. It follows that $x_{1}x_{5}$ is incident with the outer face $f_{0}$. By \ref{R5}, $f$ receives at least $\frac{1}{3}$ from $\{x, y\}$. By \ref{R6}, $f$ receives $\frac{1}{3}$ from the outer face $f_{0}$. Thus, $\tau(\rightarrow f) \geq 2 \times \frac{1}{3} > \frac{1}{2}$. So we may assume that $d(x_{5}) \geq 3$ and $d(x_{6}) \geq 3$. Note that $x_{4}$ is a $3$-vertex. By \autoref{SPECIAL6FACE1}, $x_{1}x_{5}$ cannot be incident with a $6^-$-face. That is, $x_{1}x_{5}$ is incident with a $7^{+}$-face which sends $\frac{1}{3}$ to $f$. On the other hand, by \autoref{TC1}, one vertex in $\{x_{2}, x_{3}, x_{6}, x_{7}\}$ is a special vertex which sends at least $\frac{1}{6}$ to $f$. Thus, $\tau(\rightarrow f) \geq \frac{1}{3} + \frac{1}{6} = \frac{1}{2}$. 

\textbf{Case 2.} Let $[x_{1}x_{2}x_{6}]$ and $[x_{3}x_{4}x_{8}]$ be internal $3$-faces, and let $x_{2}$ and $x_{5}$ be normal $3$-vertices. If $x_{6}$ is a $2$-vertex, then $x_{6} \in \{x, y\}$. Since $x_{2}$ is a normal vertex, $\{x, y\} = \{x_{1}, x_{6}\}$. Thus, $f$ receives $\frac{1}{3}$ from each of $x_{1}$ and $x_{6}$ by \ref{R5}, and thus $\tau(\rightarrow f) \geq \frac{1}{3} + \frac{1}{3} \geq \frac{1}{2}$. Assume $d(x_{6}) \geq 3$. By \autoref{TC1}, at least one of $x_{1}$ and $x_{6}$ is a special vertex. By \ref{R3}, \ref{R4} and \ref{R5}, $f$ receives at least $\frac{1}{6}$ from these special vertices. If $x_{3}$ is a $3$-vertex, then $x_{3} \in \{x, y\}$ by \autoref{a3}. By \ref{R5}, $f$ receives $\frac{1}{3}$ from $x_{3}$. Thus, $\tau(\rightarrow f) \geq \frac{1}{6} + \frac{1}{3} = \frac{1}{2}$. So we may assume that $d(x_{3}) \geq 4$. By \autoref{GOOD}, $x_{2}x_{3}$ is incident with a $7^{+}$-face. By \ref{RRR7}, $f$ receives $\frac{1}{3}$ from each adjacent $7^{+}$-face. Thus, $\tau(\rightarrow f) \geq \frac{1}{6} + \frac{1}{3} = \frac{1}{2}$. 
\end{proof}

\begin{claim}\label{3F}
If $f$ is an internal $5$-face and $s_f=3$, then $\tau(\rightarrow f) \geq \frac{2}{3}$.
\end{claim}
\begin{proof}
Assume $f = [x_{1}x_{2}x_{3}x_{4}x_{5}]$. According to symmetry, we only need to consider two cases.

\textbf{Case 1.} Let $[x_{1}x_{2}x_{6}]$, $[x_{2}x_{3}x_{7}]$ and $[x_{4}x_{5}x_{9}]$ be internal $3$-faces. Assume $d(x_{6}) = 2$. By \autoref{a3}, $\{x, y\} = \{x_{1}, x_{6}\}$ or $\{x, y\} = \{x_{2}, x_{6}\}$. By \ref{R5}, $f$ receives $\frac{1}{3}$ from each of $x$ and $y$, thus $\tau(\rightarrow f) \geq 2 \times \frac{1}{3} = \frac{2}{3}$. So we may assume that $d(x_{6}) \geq 3$. Similarly, we can assume that $d(x_{7}) \geq 3$ and $d(x_{9}) \geq 3$. It is clear that neither $x_{1}x_{5}$ nor $x_{3}x_{4}$ is incident with a $4^-$-face. By \autoref{FIVEFACE}, neither $x_{1}x_{5}$ nor $x_{3}x_{4}$ is incident with a $5$-face. By \autoref{SIXFACE}, neither $x_{1}x_{5}$ nor $x_{3}x_{4}$ is incident with a $6$-face. Hence, $f$ is adjacent to two $7^{+}$-faces. By \ref{R6} and \ref{RRR7}, $\tau(\rightarrow f) \geq 2 \times \frac{1}{3} = \frac{2}{3}$. 

\textbf{Case 2.} Let $[x_{1}x_{2}x_{6}]$, $[x_{2}x_{3}x_{7}]$ and $[x_{3}x_{4}x_{8}]$ be internal $3$-faces. If $d(x_{i}) = 2$ for $i\in\{5, 6, 7, 8\}$, then $x_{i}\in\{x, y\}$ by \autoref{delta}. Since $x$ and $y$ are adjacent, we have that $\{x, y\} \subset \{x_{1}, x_{2}, \dots, x_{8}\}$. By \ref{R5}, $f$ receives $\frac{1}{3}$ from each of $x$ and $y$, thus $\tau(\rightarrow f) \geq 2 \times \frac{1}{3} = \frac{2}{3}$. So we may assume that $x_{5}, x_{6}, x_{7}$ and $x_{8}$ are all $3^{+}$-vertices. It is clear that neither $x_{4}x_{5}$ nor $x_{1}x_{5}$ is contained in a $4^-$-face. By \autoref{FIVEFACE}, neither $x_{1}x_{5}$ nor $x_{4}x_{5}$ is incident with a $5$-face. Recall that $x_{6}$ is a $3^{+}$-vertex and $x_{4}x_{5}$ is not contained in a triangle. By \autoref{SIXFACE}, $x_{1}x_{5}$ cannot be incident with a $6$-face. Hence, $x_{1}x_{5}$ is incident with a $7^{+}$-face. By symmetry, $x_{4}x_{5}$ is also incident with a $7^{+}$-face. By \ref{R6} and \ref{RRR7}, $\tau(\rightarrow f) \geq 2 \times \frac{1}{3} = \frac{2}{3}$. 
\end{proof}

Now we calculate the final charge of internal $5^+$-faces. Let $f = [v_{1}v_{2} \dots v_{d}]$ be an internal $d$-face for $d\ge5$. By \autoref{TWOCONNECTED}, every face in $G$ is bounded by a cycle. Since there are no $9$-cycles, $d\ne9$. 

If $v_{i}$ is a $2$-vertex, then $v_{i} \in \{x, y\}$ and $v_{i-1}v_{i}v_{i+1}$ is incident with the outer face $f_{0}$. Thus, $f$ is adjacent to at most $d - 2$ internal faces. By \autoref{cor-1}, $t_f\le \frac{d}{2}$. By \ref{R1} and \ref{RRR7}, $f$ sends at most $\frac{1}{3}$ to each adjacent internal face. By \ref{R2}, $f$ sends at most $\frac{1}{2}$ to each incident normal $3$-vertex. By \ref{R5}, $f$ receives $\frac{1}{3}$ from each of $x$ and $y$. By \ref{R6}, $f$ receives $\frac{1}{3}$ via each of $v_{i-1}v_{i}$ and $v_{i}v_{i+1}$. Hence, 
$\mu'(f)\geq d - 4 + 4 \times \frac{1}{3} - (d - 2) \times \frac{1}{3} -\frac{d}{2} \times \frac{1}{2} > 0$.

So we may assume that there is no $2$-vertex incident with $f$.

$\bullet$ $\bm{d=5}$. 

By \autoref{cor-1} and \autoref{TRIANGULAREDGE}, $t_f\le2$ and $s_f\le3$. If $s_f=0$, then $\mu'(f) \geq 5 - 4 - 2 \times \frac{1}{3} > 0$ by \ref{R2}. 

If $s_f=1$, then $f$ is incident with at most one minor $3$-vertex. By \autoref{1F2V1M}, \ref{R1} and \ref{R2}, $\mu'(f) \geq 5 - 4 + \frac{1}{3} - \frac{1}{3} - \left(\frac{1}{2} + \frac{1}{3}\right) > 0$.

Assume $s_f=2$. If $t_f=0$, then $\mu'(f) \geq 5 - 4 - 2 \times \frac{1}{3} > 0$ by \ref{R1}. Let $t_f=1$. If the normal $3$-vertex is not minor, then $\mu'(f) \geq 5 - 4 - 2 \times \frac{1}{3} - \frac{1}{3} = 0$ by \ref{R1} and \ref{R2}. If the normal $3$-vertex is minor, then $\mu'(f) \geq 5 - 4 + \frac{1}{3} - 2 \times \frac{1}{3} - \frac{1}{2} > 0$ by \autoref{2F1M}, \ref{R1} and \ref{R2}. Let $t_f=2$. It is observed that $f$ is incident with at least one minor $3$-vertex. If $f$ is incident with exactly one minor $3$-vertex, then $\mu'(f) \geq 5 - 4 + \frac{1}{2} - 2 \times \frac{1}{3} - (\frac{1}{2} + \frac{1}{3}) = 0$ by \autoref{2F2V1M}, \ref{R1} and \ref{R2}. The other situation, $f$ is incident with exactly two minor $3$-vertices. Thus, $\mu'(f) \geq 5 - 4 + 1 - 2 \times \frac{1}{3} - 2 \times \frac{1}{2} > 0$ by \autoref{2F2M}, \ref{R1} and \ref{R2}. 

Assume $s_f=3$. If $t_f=0$, then $\mu'(f) \geq 5 - 4 - 3 \times \frac{1}{3} = 0$ by \ref{R1}. If $t_f=1$, then $\mu'(f)\geq 5 - 4 + \frac{2}{3} - 3 \times \frac{1}{3} - \frac{1}{2} > 0$ by \autoref{3F}, \ref{R1} and \ref{R2}. If $t_f=2$, then it is incident with two minor $3$-vertices, and then $\mu'(f) \geq 5 - 4 + 1 - 3 \times \frac{1}{3} - 2 \times \frac{1}{2} = 0$ by \autoref{2F2M}, \ref{R1} and \ref{R2}. 

$\bullet$ $\bm{d=6}$. 

Note that there are no $4$-cycle in $G$. If $f$ is adjacent to a $3$-face, then it must be normally adjacent to the $3$-face. Since there are no $9$-cycles in $G$, $f$ is adjacent to at most two $3$-faces. It follows that $f$ is incident with at most two minor $3$-vertices. By \ref{R1} and \ref{R2}, $\mu'(f) \geq 6 - 4 - 2 \times \frac{1}{3} - (2 \times \frac{1}{2} + \frac{1}{3}) = 0$. 

$\bullet$ $\bm{d=7}$. 

If $f$ is adjacent to a $3$-face, then it must be normally adjacent to the $3$-face. Otherwise, there is a $4$-cycle in $G$. Since there are no $9$-cycles in $G$, $f$ is adjacent to at most one $3$-face. It follows that $f$ is incident with at most one minor $3$-vertex. By \autoref{cor-1}, $t_f\le3$. If $t_f=3$, then $f$ is adjacent to at most four good $5$-faces by \autoref{a3} and \autoref{GOODFACE}, and then $\mu'(f) \geq 7 - 4 - (1 + 4) \times \frac{1}{3} - (\frac{1}{2} + 2 \times \frac{1}{3}) > 0$ by \ref{R1}, \ref{R2} and \ref{RRR7}. If $t_f=2$, then $f$ is adjacent to at most five good $5$-faces by \autoref{a3} and \autoref{GOODFACE}, and then $\mu'(f) \geq 7 - 4 - (1 + 5) \times \frac{1}{3} - (\frac{1}{2} + \frac{1}{3}) > 0$ by \ref{R1}, \ref{R2} and \ref{RRR7}. If $t_f=1$, then $f$ is adjacent to at most six good $5$-faces by \autoref{a3} and \autoref{GOODFACE}, and then $\mu'(f) \geq 7 - 4 - (1 + 6) \times \frac{1}{3} - \frac{1}{2} > 0$ by \ref{R1}, \ref{R2} and \ref{RRR7}. If $t_f=0$, then $\mu'(f) \geq 7 - 4 - 7 \times \frac{1}{3} > 0$ by \ref{R1} and \ref{RRR7}. 

$\bullet$ $\bm{d=8}$. 

Similar to the above cases, if $f$ is adjacent to a $3$-face, then it must be normally adjacent to the $3$-face. Since there are no $9$-cycles, $f$ is not adjacent to any $3$-face. Thus, $f$ is not incident with any minor $3$-vertex. By \ref{R2} and \ref{RRR7}, $\mu'(f) \geq 8 - 4 - 8 \times \frac{1}{3} - 4 \times \frac{1}{3} = 0$. 

$\bullet$ $\bm{d\ge10}$. 

By \ref{R1} and \ref{RRR7}, $f$ sends at most $\frac{1}{3}$ via each incident edge. It follows that $\mu'(f) \geq d - 4 - d \times \frac{1}{3} - \frac{d}{2} \times \frac{1}{2} > 0$. 

This completes the proof of \autoref{HP}.

\end{document}